\documentclass[oneside]{amsart}
\usepackage[latin9]{inputenc}
\usepackage{amsthm}
\usepackage{graphicx}

\usepackage{amsmath, amsthm, amssymb,mathtools,slashbox,placeins,aliascnt,setspace,fancybox,hyperref}
\usepackage[all]{hypcap}
\usepackage{fullpage}

\theoremstyle{plain}
\newtheorem{thm}{Theorem}

\theoremstyle{definition}
\newaliascnt{defn}{thm}
\newtheorem{defn}[defn]{Definition}
\aliascntresetthe{defn}

\theoremstyle{remark}
\newaliascnt{rem}{thm}
\newtheorem{rem}[rem]{Remark}
\aliascntresetthe{rem}

\theoremstyle{plain}
\newaliascnt{fact}{thm}

\aliascntresetthe{fact}

\theoremstyle{plain}
\newaliascnt{lem}{thm}
\newtheorem{lem}[lem]{Lemma}
\aliascntresetthe{lem}

\theoremstyle{plain}
\newaliascnt{prop}{thm}
\newtheorem{prop}[prop]{Proposition}
\aliascntresetthe{prop}

\theoremstyle{plain}
\newaliascnt{proc}{thm}

\aliascntresetthe{proc}

\theoremstyle{plain}
\newaliascnt{cor}{thm}
\newtheorem{cor}[cor]{Corollary}
\aliascntresetthe{cor}

\theoremstyle{definition}
\newaliascnt{example}{thm}

\aliascntresetthe{example}

\theoremstyle{plain}
\newaliascnt{conv}{thm}
\newtheorem{conv}[conv]{Convention}
\aliascntresetthe{conv}

\def\equationautorefname~#1\null{(#1)\null}

\newcommand{\A}{\mathcal A} % A columns
\newcommand{\B}{\mathcal B} % B columns
\newcommand{\C}{\mathcal C} % C columns
\newcommand{\D}{\mathcal D} % D columns
\newcommand{\F}{\mathcal F} % non-empty columns
\renewcommand{\H}{\mathcal H} % arrow columns

\newcommand{\X}{\mathcal X} % generic columns
\newcommand{\Y}{\mathcal Y} % generic columns
\newcommand{\Z}{\mathcal Z} % generic columns

\newcommand{\K}{\mathcal K} % all columns
\newcommand{\R}{\mathcal R} % marked columns
\renewcommand{\P}{\mathcal P} % fix-point free involutions
\newcommand{\AP}{\widetilde{\mathcal A}}

\newcommand{\CP}{\widetilde{\mathcal C}}

\newcommand{\XP}{\widetilde{\mathcal X}}

\newcommand{\ap}{\widetilde a}
\newcommand{\cp}{\widetilde c}

\newcommand{\xp}{\widetilde x}

\newcommand{\ab}{\overline a}
\newcommand{\cb}{\overline c}

\newcommand{\xb}{\overline x}

\newcommand{\AB}{\overline{\mathcal A}}

\newcommand{\CB}{\overline{\mathcal C}}

\newcommand{\XB}{\overline{\mathcal X}}

\newcommand{\PA}[4]{\mathcal{PA}_{#1;#2}^{\left(#3;#4\right)}}
\newcommand{\CA}[3]{\mathcal{CA}_{#1}^{\left(#2;#3\right)}}
\newcommand{\VA}[3]{\mathcal{VA}_{#1;#2}^{\left(#3\right)}}
\newcommand{\PVA}[3]{\mathcal{PVA}_{#1;#2}^{\left(#3\right)}}
\newcommand{\AR}[4]{\mathcal{AR}_{#1;#2,#3}^{\left(#4\right)}}

\begin{document}
\let\ref\autoref

\title{Methods of Enumerating Two Vertex Maps of Arbitrary Genus}

\author{Aaron Chun Shing Chan}

\date{December 5, 2016}
\begin{abstract}
This paper provides an alternate proof to parts of the Goulden-Slofstra
formula \cite{Goulden-Slofstra:2010} for enumerating two vertex maps
by genus, which is an extension of the famous Harer-Zagier formula
\cite{Harer-Zagier:1986} that computes the Euler characteristic of
the moduli space of curves. This paper also shows a further simplification
to the Goulden-Slofstra formula. Portions of this alternate proof
will be used in a subsequent paper \cite{Chan:2017-2}, where it forms
the basis for a more general result that applies for a certain class
of maps with an arbitrary number of vertices.
\end{abstract}

\maketitle

\section{\label{sec:Introduction}Introduction}

Let $S$ be a set of even cardinality. A \emph{pairing} $\mu$ of
$S$ is a partition of $S$ into disjoint subsets of size 2. In the
context of permutations, $\mu$ can be seen as a \emph{fixed-point
free involution}, where every cycle of $\mu$ is of size 2. Now, let
$p=2q$ be a positive integer. We use $\left[p\right]$ to denote
the set $\left\{ 1,\dots,p\right\} $, and $\P_{p}$ to be the set
of all pairings of $\left[p\right]$. If $\gamma_{p}$ is the canonical
cycle permutation of $\left[p\right]$, given by $\gamma_{p}=\left(1,\dots,p\right)$,
we have the following theorem by Harer-Zagier on the Euler characteristic
of the moduli space of curves.
\begin{thm}
(Harer-Zagier \cite{Harer-Zagier:1986}) \label{thm:Harer-Zagier}
Let $q$ be a positive integer, and $\A_{L}^{\left(q\right)}$ be
the subset of pairings of $\P_{2q}$ such that for $\mu\in\A_{L}^{\left(q\right)}$,
$\mu\gamma_{2q}^{-1}$ has exactly $L$ cycles. If we let $a_{L}^{\left(q\right)}=\left|\A_{L}^{\left(q\right)}\right|$,
then the generating series for $a_{L}^{\left(q\right)}$ is given
by 
\[
A^{\left(q\right)}\left(x\right)=\left(2q-1\right)!!\sum_{k\ge1}2^{k-1}\binom{q}{k-1}\binom{x}{k}
\]
where $\left(2k-1\right)!!=\prod_{j=1}^{k}\left(2j-1\right)$ is the
double factorial.
\end{thm}
There are numerous proofs of this formula in the literature, both
algebraic and combinatorial. A selection of the proofs can be found
in the papers by Goulden and Nica \cite{Goulden-Nica:2005}, Itzykson
and Zuber \cite{Itzykson-Zuber:1990}, Jackson \cite{Jackson:1994},
Kerov \cite{Kerov:1999}, Kontsevich \cite{Kontsevich:1992}, Lass
\cite{Lass:2001}, Penner \cite{Penner:1988}, and Zagier \cite{Zagier:1995}.
As seen in Lando and Zvonkin \cite{Lando-Zvonkin:2004}, the Harer-Zagier
formula enumerates 1-celled embeddings on an orientable surface by
genus, which are equivalent to one vertex maps with $q$ loop edges.
The original proof of Harer-Zagier uses matrix integration, and there
are numerous other algebraic proofs for this same result. Some subsequent
proofs used purely combinatorial approaches, such as the use of Eulerian
tours by Lass, and the use of trees by Goulden and Nica.

Next, we will set up the terminology for the Goulden and Slofstra
result, which is an extension of the Harer-Zagier formula. Let $p,n\ge1$.
We use $\left[p\right]^{\underline{n}}$ to denote the set $\left\{ 1^{^{\underline{n}}},2^{^{\underline{n}}},\dots,p^{^{\underline{n}}}\right\} $,
whose elements $i^{\underline{n}}$, $i=1,\dots,p$, are regarded
as a labelled version of the integer $i$, labelled by the ``$\underline{n}$''
in the superscript position. Then, suppose $p_{1}$ and $p_{2}$ are
positive integers, we let $\left[p_{1},p_{2}\right]$ to be the set
$\left[p_{1}\right]^{\underline{1}}\cup\left[p_{2}\right]^{\underline{2}}$.
For example, $\left[3,5\right]$ is the set $\left\{ 1^{^{\underline{1}}},2^{^{\underline{1}}},3^{^{\underline{1}}},1^{^{\underline{2}}},2^{^{\underline{2}}},3^{^{\underline{2}}},4^{^{\underline{2}}},5^{^{\underline{2}}}\right\} $.
Furthermore, if $p_{1}+p_{2}$ is even, then the set of all pairings
of $\left[p_{1},p_{2}\right]$ is denoted as $\P_{p_{1},p_{2}}$.
Now, if $\mu$ is a pairing of $\left[p_{1},p_{2}\right]$, then a
pair $\left\{ x^{\underline{i}},y^{\underline{k}}\right\} $ in $\mu$
is a \emph{mixed pair} if $i\neq k$, and a \emph{non-mixed pair}
otherwise. To describe the number of mixed and non-mixed pairs in
a pairing $\mu$, we introduce the parameters $q_{1}$, $q_{2}$,
and $s$. Let $q_{1},q_{2}\ge0$ and $s>0$ such that $p_{i}=2q_{i}+s$
for $i=1,2$. We define $\P^{\left(q_{1},q_{2};s\right)}\subseteq\P_{p_{1},p_{2}}$
to be the subset of the pairing such that for $\mu\in\P^{\left(q_{1},q_{2};s\right)}$,
$\mu$ has $q_{i}$ non-mixed pairs of the form $\left\{ x^{\underline{i}},y^{\underline{i}}\right\} $
and $s$ mixed pairs. If $\gamma_{p_{1},p_{2}}$ is the canonical
cycle permutation of $\left[p_{1},p_{2}\right]$, given by $\gamma_{p_{1},p_{2}}=\left(1^{^{\underline{1}}},\dots,p_{1}^{^{\underline{1}}}\right)\left(1^{^{\underline{2}}},\dots,p_{2}^{^{\underline{2}}}\right)$,
then the series that enumerates the number of two vertex maps according
to the genus is given as follows.
\begin{thm}
(Goulden-Slofstra \cite{Goulden-Slofstra:2010}) \label{thm:Goulden-Slofstra}
Let $q_{1}$ and $q_{2}$ be non-negative integers, and $s$ be a
positive integer. Let $\A_{L}^{\left(q_{1},q_{2};s\right)}$ be the
subset of pairings of $\P^{\left(q_{1},q_{2};s\right)}$ such that
for $\mu\in\A_{L}^{\left(q_{1},q_{2};s\right)}$, $\mu\gamma_{p_{1},p_{2}}^{-1}$
has exactly $L$ cycles. If we let $a_{L}^{\left(q_{1},q_{2};s\right)}=\left|\A_{L}^{\left(q_{1},q_{2};s\right)}\right|$,
then the generating series for $a_{L}^{\left(q_{1},q_{2};s\right)}$
is given by 
\[
A^{\left(q_{1},q_{2};s\right)}\left(x\right)=p_{1}!p_{2}!\sum_{k=1}^{d+1}\sum_{i=0}^{\left\lfloor \frac{1}{2}p_{1}\right\rfloor }\sum_{j=0}^{\left\lfloor \frac{1}{2}p_{2}\right\rfloor }\frac{1}{2^{i+j}i!j!\left(d-i-j\right)!}\binom{x}{k}\binom{d-i-j}{k-1}\Delta_{k}^{\left(q_{1},q_{2};s\right)}
\]
where $p_{1}=2q_{1}+s$, $p_{2}=2q_{2}+s$, $d=q_{1}+q_{2}+s$, and
\[
\Delta_{k}^{\left(q_{1},q_{2};s\right)}=\binom{k-1}{q_{1}-i}\binom{k-1}{q_{2}-j}-\binom{k-1}{q_{1}+s-i}\binom{k-1}{q_{2}+s-j}
\]
\end{thm}
In this expression, $p_{1}$ and $p_{2}$ are the degrees of vertices
1 and 2, respectively, and $d$ is the total number of pairs in the
pairing. Similar to the Harer-Zagier formula, the Goulden-Slofstra
formula counts the number of combinatorial maps with 2 vertices by
genus, where there are $q_{1}$ and $q_{2}$ loop edges on vertices
1 and 2, and $s$ edges between the two vertices. To represent these
maps, Goulden and Slofstra used a combinatorial object called the
paired surjections, which we will define in the next section.

\section{\label{sec:Paired Arrays}Definitions and Terminology of Paired Arrays}

In this section, we will mostly follow the methodology of Goulden
and Slofstra \cite{Goulden-Slofstra:2010}. For that reason, we will
not be providing proofs for the results stated, and skip over some
of their constructions. However, we will be defining some terminology
of our own, so that we can extend their approach later. Note that
our notation in this paper is generally different from that of Goulden
and Slofstra, as it makes it easier to refer to the results in the
follow up paper \cite{Chan:2017-2} that covers multiple vertices.
\begin{defn}
\label{def:Paired Functions}Let $K,s\ge1$, $q_{1},q_{2}\ge0$, and
$p_{i}=2q_{i}+s$ for $i=1,2$. An ordered pair $\left(\mu,\pi\right)$
is a \emph{paired surjection} if $\mu\in\P^{\left(q_{1},q_{2};s\right)}$
and $\pi\colon\left[p_{1},p_{2}\right]\rightarrow\left[K\right]$
is a surjection satisfying 
\[
\pi\left(\mu\left(v\right)\right)=\pi\left(\gamma_{p_{1},p_{2}}\left(v\right)\right)\quad\mbox{ for all }v\in\left[p_{1},p_{2}\right]
\]
We denote the set of paired surjection satisfying the parameters $K$,
$q_{1}$, $q_{2}$, and $s$ as $\F_{K}^{\left(q_{1},q_{2};s\right)}$,
and we let $f_{K}^{\left(q_{1},q_{2};s\right)}=\left|\F_{K}^{\left(q_{1},q_{2};s\right)}\right|$.
\end{defn}
We can then express the generating series $A^{\left(q_{1},q_{2};s\right)}\left(x\right)$
using paired surjections as follows.
\begin{prop}
(Goulden-Slofstra \cite{Goulden-Slofstra:2010}) \label{prop:Paired surjection formula}For
$q_{1},q_{2}\ge0$ and $s\ge1$, we have
\begin{eqnarray*}
A^{\left(q_{1},q_{2};s\right)}\left(x\right) & = & \sum_{K\ge1}f_{K}^{\left(q_{1},q_{2};s\right)}\binom{x}{K}
\end{eqnarray*}
\end{prop}
Now, paired surjections can be represented graphically with a combinatorial
object called the \emph{labelled array}. This is an $2\times K$ array
of cells arranged in a grid. Each element $x^{\underline{i}}$ of
$\mu$ is represented as a vertex, where the vertex labelled $x^{\underline{i}}$
is placed into cell $\left(i,j\right)$ if $\pi\left(x^{\underline{i}}\right)=j$.
The vertices are arranged horizontally within a cell, in increasing
order of the labels. Furthermore, for each pair $\left\{ x^{\underline{i}},y^{\underline{k}}\right\} $
in $\mu$, an edge is drawn between their corresponding vertices.

For example, let $\left(\mu,\pi\right)\in\F_{4}^{\left(3,1;4\right)}$,
with $\mu$ and $\pi$ given by 
\begin{eqnarray*}
\mu & = & \left\{ \left\{ 1^{^{\underline{1}}},4^{^{\underline{2}}}\right\} ,\left\{ 2^{^{\underline{1}}},3^{^{\underline{1}}}\right\} ,\left\{ 4^{^{\underline{1}}},3^{^{\underline{2}}}\right\} ,\left\{ 5^{^{\underline{1}}},7^{^{\underline{1}}}\right\} ,\left\{ 6^{^{\underline{1}}},1^{^{\underline{2}}}\right\} ,\left\{ 8^{^{\underline{1}}},5{}^{^{\underline{2}}}\right\} ,\left\{ 9^{^{\underline{1}}},10^{^{\underline{1}}}\right\} ,\left\{ 2^{^{\underline{2}}},6^{^{\underline{2}}}\right\} \right\} 
\end{eqnarray*}
\begin{align*}
\pi^{-1}\left(1\right) & =\left\{ 2^{^{\underline{1}}},4^{^{\underline{1}}},4^{^{\underline{2}}}\right\}  & \pi^{-1}\left(2\right) & =\left\{ 3^{^{\underline{1}}},5^{^{\underline{1}}},8^{^{\underline{1}}},3^{^{\underline{2}}},6^{^{\underline{2}}}\right\} \\
\pi^{-1}\left(3\right) & =\left\{ 1^{^{\underline{1}}},9^{^{\underline{1}}},10^{^{\underline{1}}},5^{^{\underline{2}}}\right\}  & \pi^{-1}\left(4\right) & =\left\{ 6^{^{\underline{1}}},7^{^{\underline{1}}},1^{^{\underline{2}}},2^{^{\underline{2}}}\right\} 
\end{align*}
Then, the labelled array representing $\left(\mu,\pi\right)$ is given
by \ref{fig:Labelled Array}.

\begin{figure}
\begin{centering}
\includegraphics{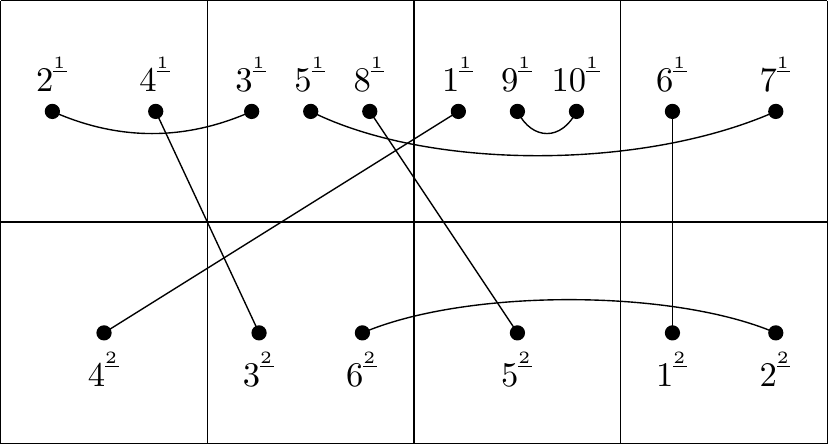}
\par\end{centering}
\caption{\label{fig:Labelled Array}A labelled array with 4 columns}
\end{figure}

Note that a $2\times K$ array with paired and labelled vertices as
described above uniquely represents a pairing $\mu\in\P^{\left(q_{1},q_{2};s\right)}$
and a function $\pi\colon\left[p_{1},\dots,p_{n}\right]\rightarrow\left[K\right]$.
Furthermore, we can strip the labels and define paired arrays as abstract
combinatorial objects, with conditions that allow for a bijection
between paired arrays and labelled arrays.
\begin{defn}
\label{def:Paired Array}Let $K,s\ge1$, $q_{1},q_{2}\ge0$, and $1\le R_{1},R_{2}\le K$.
We define $\PA{K}{R_{1},R_{2}}{q_{1},q_{2}}{s}$ to be the set of
\emph{paired arrays}, which are arrays of cells and vertices subject
to the following conditions.

\begin{itemize}
\item A paired array is an $2\times K$ array of cells, such that each cell
$\left(i,j\right)$ contains an ordered list of vertices, arranged
left to right, so that row $i$ contains $p_{i}\coloneqq2q_{i}+s$
vertices for $i=1,2$.
\item Each vertex $u$ is paired with exactly one other vertex $v$. Exactly
$2q_{i}$ vertices of row $i$ are paired with other vertices of row
$i$, and exactly $s$ vertices of row $i$ are paired with vertices
of the other row. Graphically, the pairings are denoted as edges between
vertices.
\item Each row $i$ has exactly $R_{i}$ marked cells, which are denoted
by marking the cell with a box in its upper or lower right corner.
\item A pair of vertices $\left\{ u,v\right\} $ is a \emph{mixed pair}
if $u$ and $v$ belong to different rows. The vertices $u$ and $v$
are called \emph{mixed vertices}.
\end{itemize}
\end{defn}
Generally, we use $\alpha\in\PA{K}{R_{1},R_{2}}{q_{1},q_{2}}{s}$
to denote a paired array. Before introducing the conditions used in
Goulden and Slofstra, we will first introduce a number of useful notations
and conventions.
\begin{conv}
\label{conv:Array Convention}For notational convenience, we introduce
the following:

\begin{itemize}
\item We use calligraphic letters to denote columns or sets of columns.
For generic columns or sets of columns, we use the letters $\X$,
$\Y$, and $\Z$.
\item For each calligraphic letter, we use the corresponding upper case
letter to denote the number of columns in the set. For example, $X=\left|\X\right|$.
\item For each calligraphic letter, we use the corresponding lower case
letter, subscripted by the row number, to denote the total number
of vertices in those columns for a given row. For example, $x_{i}$
is the total number of vertices in row $i$ of the columns of $\X$.
\item We generally use $i,j,k,\ell$ as index variables, with $i$ and $k$
for rows, and $j$ and $\ell$ for columns. Furthermore, we use cell
$\left(i,j\right)$ to denote the cell in row $i$, column $j$ of
the array.
\item We use $\K$ to denote the set of all columns, and $K$ to denote
the total number of columns.
\item We use $\R_{i}$ to denote the set of columns that are marked in row
$i$, and $R_{i}$ to denote the number of columns that are marked
in row $i$.
\item We use $\F_{i}$ to denote the set of columns that have at least one
vertex in row $i$, and $F_{i}$ to denote the number of columns that
are marked in row $i$.
\item We use $w_{i,j}$ to denote the number of vertices in cell $\left(i,j\right)$,
and $\mathbf{w}$ to denote a matrix of $w_{i,j}$ describing the
number of vertices in each cell of row $i$.
\end{itemize}
\end{conv}
With these conventions, we are ready to define the three conditions
that allow us to create a bijection between labelled arrays and paired
arrays.
\begin{defn}
\label{def:Paired Array Conditions}Let $\alpha\in\PA{K}{R_{1},R_{2}}{q_{1},q_{2}}{s}$
be a paired array.

\begin{itemize}
\item $\alpha$ is said to satisfy the \emph{non-empty condition} if each
column $j$ contains at least one object.
\item $\alpha$ is said to satisfy the \emph{balance condition} if for each
column $j$, the number of mixed vertices in cell $\left(1,j\right)$
and cell $\left(2,j\right)$ are equal.
\item For rows $i=1,2$, the \emph{forest condition function} $\psi_{i}\colon\F_{i}\backslash\R_{i}\mapsto\K$
is defined as follows: For each column $j\in\F_{i}\backslash\R_{i}$,
if the rightmost vertex $v$ is paired with a vertex $u$ in column
$\ell$, then $\psi_{i}\left(j\right)=\ell$. $\alpha$ is said to
satisfy the \emph{forest condition} if for each row $i$, the functional
digraph of $\psi_{i}$ on the vertex set $\F_{i}\cup\psi_{i}\left(\F_{i}\right)\cup\R_{i}$
is a forest with root vertices $\R_{i}$. That is, for each column
$j\in\F_{i}\backslash\R_{i}$, there exists some positive integer
$t$ such that $\psi_{i}^{t}\left(j\right)\in\R_{i}$. Note that we
always include $\R_{i}$ in the vertex set of the functional digraph
of $\psi_{i}$, regardless of whether they are in the domain or range
of $\psi_{i}$.
\end{itemize}
A paired array is \emph{proper} if it satisfies the non-empty, balance,
and forest conditions, and a paired array is a \emph{canonical array}
if it is proper and $R_{1}=R_{2}=1$. We denote the set of canonical
arrays as $\CA{K}{q_{1},q_{2}}{s}$, and we let $c_{K}^{\left(q_{1},q_{2};s\right)}=\left|\CA{K}{q_{1},q_{2}}{s}\right|$.
A paired array is called a \emph{vertical array} if for every pair
$\left\{ u,v\right\} $, $u$ and $v$ are in different rows, and
is \emph{proper} if it satisfies the non-empty, balance, and forest
conditions. We denote the set of vertical arrays as $\VA{K}{R_{1},R_{2}}{s}=\PA{K}{R_{1},R_{2}}{0,0}{s}$
and the set of proper vertical arrays as $\PVA{K}{R_{1},R_{2}}{s}$.
Finally, we let $v_{K;R_{1},R_{2}}^{\left(s\right)}=\left|\PVA{K}{R_{1},R_{2}}{s}\right|$.
\end{defn}
Note that we will generally not work directly with paired arrays that
do not satisfy the forest condition. However, as vertical arrays not
satisfying the forest condition are vital for extending paired arrays,
we have separated the forest condition from the definition of vertical
arrays itself. Next, we give a formula for relating the number of
canonical arrays to the number of vertical arrays.

\begin{figure}
\begin{centering}
\includegraphics{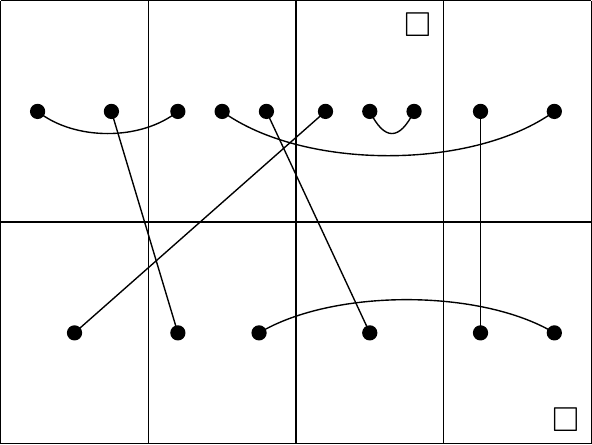}
\par\end{centering}
\caption{\label{fig:Canonical Array}A canonical paired array with 4 columns}
\end{figure}

\begin{thm}
(Goulden-Slofstra \cite{Goulden-Slofstra:2010}) \label{thm:Label to Canonical Array}For
$K,s\ge1$ and $q_{1},q_{2}\ge0$, we have $f_{K}^{\left(q_{1},q_{2};s\right)}=c_{K}^{\left(q_{1},q_{2};s\right)}$.
\end{thm}
To obtain a canonical array from a labelled array, we simply marked
the cells that contain 1 in both rows, then delete the labels. Applying
this to the labelled array in \ref{fig:Labelled Array} gives us the
canonical array in \ref{fig:Canonical Array}. With this result, the
problem of enumerating maps on surfaces reduces to that of enumerating
canonical arrays. To solve the latter problem, we will first decompose
canonical arrays by removing all non-mixed pairs using the following
theorem.
\begin{thm}
(Goulden-Slofstra \cite{Goulden-Slofstra:2010}) \label{thm:Canonical Vertical Formula}Let
$n,K,s\ge1$ and $q_{1},q_{2}\ge0$. We have
\[
c_{K}^{\left(q_{1},q_{2};s\right)}=\sum_{t_{1},t_{2}\ge0}\prod_{i=1}^{n}\frac{p_{1}!p_{2}!}{2^{t_{1}+t_{2}}t_{1}!t_{2}!\left(s_{1}+q_{1}-t_{1}\right)!\left(s_{2}+q_{2}-t_{2}\right)!}\cdot v_{K;q_{1}-t_{1}+1,q_{2}-t_{2}+1}^{\left(s\right)}
\]
\end{thm}
For example, by decomposing the canonical array in \ref{fig:Canonical Array},
we can obtain the vertical array in \ref{fig:Vertical Array}. Then,
by combining the theorems we have so far, we can write the generating
series in terms of the number of vertical arrays.
\begin{cor}
(Goulden-Slofstra \cite{Goulden-Slofstra:2010}) \label{cor:Canonical Vertical Formula}Let
$n,K,s\ge1$ and $q_{1},q_{2}\ge0$. We have
\[
A^{\left(q_{1},q_{2};s\right)}\left(x\right)=\sum_{\substack{K\ge1\\
t_{1},t_{2}\ge0
}
}\binom{x}{K}\cdot\frac{p_{1}!p_{2}!}{2^{t_{1}+t_{2}}t_{1}!t_{2}!\left(s_{1}+q_{1}-t_{1}\right)!\left(s_{2}+q_{2}-t_{2}\right)!}\cdot v_{K;q_{1}-t_{1}+1,q_{2}-t_{2}+1}^{\left(s\right)}
\]
\end{cor}
\begin{rem}
While \ref{thm:Canonical Vertical Formula} is proved in Goulden and
Slofstra using the \emph{forest completion algorithm}, we can in fact
use the techniques developed in this paper to bypass this requirement
if we so desire. This alternate approach can be found in \cite{ChanThesis:2016}.
\end{rem}
\begin{figure}
\begin{centering}
\includegraphics{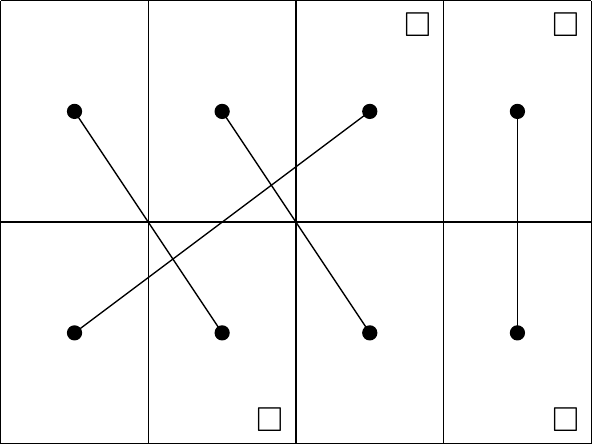}
\par\end{centering}
\caption{\label{fig:Vertical Array}Proper vertical array from the decomposition
of \ref{fig:Canonical Array}}
\end{figure}

\section{\label{sec:Arrowed Array Definitions}Definitions and Terminology
of Arrowed Arrays}

In this section, we will extend paired arrays by the addition of arrows,
which represent hypothetical vertices used in the forest condition.
This will allow us to decouple the forest condition with the vertex
pairings, which allows for the deletion of vertices and pairings from
paired arrays.
\begin{defn}
\label{def:Arrowed Array}Let $K\ge1$, $s\ge0$, and $1\le R_{1},R_{2}\le K$.
An \emph{arrowed array} is a pair $\left(\alpha,\phi\right)$, where
$\alpha\in\VA{K}{R_{1},R_{2}}{s}$ is a two-row vertical array, and
$\phi\colon\K\backslash\R_{1}\rightarrow\K$ is a partial function
from $\H\subseteq\K\backslash\R_{1}$ to $\K$, with $\R_{1}$ being
the set of marked columns in row 1 of $\alpha$. Graphically, $\phi$
is denoted by arrows drawn above row 1, where an arrow from $j$ to
$j^{\prime}$ is drawn if $j\in\H$ and $\phi\left(j\right)=j^{\prime}$.
For convenience, the two ends of the arrow belonging to columns $j$
and $j^{\prime}$ are called the \emph{arrow-tail} and \emph{arrow-head}
respectively, and column $j$ is said to \emph{point to} column $j^{\prime}$.
Furthermore, both the arrow-tail and arrow-head belong to row 1 of
their respective columns.

With the generalization of paired arrays to arrowed arrays, there
are corresponding generalizations of the terms and conventions used
to describe paired arrays. These generalizations will be compatible
with the conventions for paired arrays if the partial function $\phi$
is empty.

\begin{itemize}
\item An \emph{object} of $\left(\alpha,\phi\right)$ refers to either a
vertex, a box, or an arrow-tail. If a cell both contains vertices
and a box, or vertices and an arrow-tail, either the box or the arrow-tail
is to be taken as the rightmost object of the cell. 
\item A vertex $v$ of an arrowed array is \emph{critical }if it is the
rightmost vertex of a cell, and the cell it belongs to is neither
marked nor contains an arrow-tail. A pair $\left\{ u,v\right\} $
that contains a critical vertex is a \emph{critical pair}.
\item $\left(\alpha,\phi\right)$ is said to satisfy the \emph{non-empty
condition} if for each column $j$, there exists at least one cell
that contains an object.
\item $\left(\alpha,\phi\right)$ is said to satisfy the \emph{balance condition}
if for each column $j$, the number of vertices in cell $\left(1,j\right)$
is equal to the number of vertices in cell $\left(2,j\right)$.
\item Let $\F_{i}$ be the set of columns in row $i$ that contain at least
one vertex. The \emph{forest condition function} $\psi_{1}\colon\left(\H\cup\F_{1}\right)\backslash\R_{1}\mapsto\K$
for row 1 is defined as follows: For each column $j\in\H$, let $\psi_{1}\left(j\right)=\phi\left(j\right)$;
for $j\in\F_{1}\backslash\left(\H\cup\R_{1}\right)$, if the rightmost
vertex $v$ is paired with a vertex $u$ in column $j^{\prime}$,
let $\psi_{1}\left(j\right)=j^{\prime}$. The forest condition function
$\psi_{2}$ for row 2 is defined to be the same as the one for paired
arrays in \ref{def:Paired Array Conditions}. $\left(\alpha,\phi\right)$
is said to satisfy the \emph{forest condition} if the functional digraph
of $\psi_{1}$ on the vertex set $\H\cup\F_{1}\cup\psi_{1}\left(\H\cup\F_{1}\right)\cup\R_{1}$
is a forest with root vertices $\R_{1}$, and the functional digraph
of $\psi_{2}$ on the vertex set $\F_{2}\cup\psi_{2}\left(\F_{2}\right)\cup\R_{2}$
is a forest with root vertices $\R_{2}$. That is, for each column
$j\in\left(\H\cup\F_{1}\right)\backslash\R_{1}$, there exists some
positive integer $t$ such that $\psi_{1}^{t}\left(j\right)\in\R_{1}$,
and for each column $j\in\F_{2}\backslash\R_{2}$, there exists some
positive integer $t$ such that $\psi_{2}^{t}\left(j\right)\in\R_{2}$.
\item Additionally, $\left(\alpha,\phi\right)$ is said to satisfy the \emph{full
condition} if every cell contains at least one object.
\end{itemize}
The set of arrowed arrays that satisfy the forest condition is denoted
$\AR{K}{R_{1}}{R_{2}}{s}$.
\end{defn}
Notice in particular that a cell cannot contain both an arrow-tail
and be marked at the same time. Furthermore, a vertex is critical
if and only if it contributes to the forest condition function. Unless
otherwise stated, we will continue to use the conventions for paired
arrays defined in \ref{conv:Array Convention} for arrowed arrays.
As with paired arrays, we will always include the columns $\R_{i}$
in the vertex set for the functional digraph of $\psi_{i}$, regardless
of whether they are in the range of $\psi_{i}$. Note that permuting
the columns of an arrowed array does not change whether the array
satisfies the balance or forest conditions, as all this action does
is to relabel the vertices of the functional digraph. An example of
an arrowed array that satisfies the forest condition can be found
in \ref{fig:Arrowed Array}.

\begin{figure}
\begin{centering}
\includegraphics{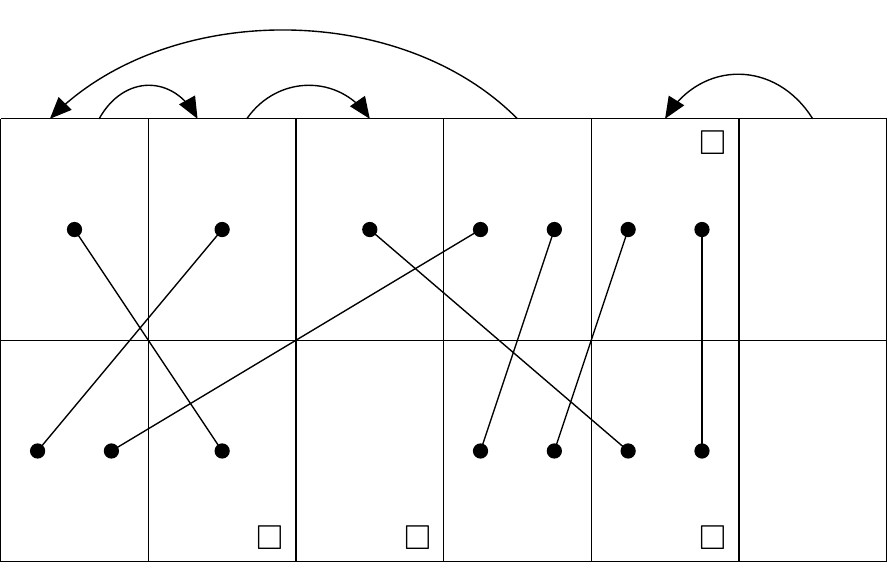}
\par\end{centering}
\caption{\label{fig:Arrowed Array}A arrowed array in $\AR{6}{1}{3}{7}$}
\end{figure}

While the parameters used for defining the set of arrowed arrays is
natural with respect to paired arrays, it does not easily lend itself
to a formula. To make it manageable for summation, we need to partition
the set of arrowed arrays by adding further constraints.
\begin{defn}
\label{def:Substructure}Let $K\ge1$, $s\ge0$, and $1\le R_{1},R_{2}\le K$.
A \emph{substructure} $\Theta$ of $\AR{K}{R_{1}}{R_{2}}{s}$ is a
set of constraints that defines a subset of $\AR{K}{R_{1}}{R_{2}}{s}$.
For convenience, an arrowed array $\left(\alpha,\phi\right)$ is said
to satisfy $\Theta$ if $\left(\alpha,\phi\right)$ satisfies the
constraints given by $\Theta$. In particular, let $\mathbf{w}$ be
a non-negative matrix of size $2\times K$, $\R_{1},\R_{2}$ be $R_{1}$
and $R_{2}$ subsets of $\K$, and $\phi$ be a partial function from
$\H\subseteq\K\backslash\R_{1}$ to $\K$. The substructure $\Gamma=\left(\mathbf{w},\R_{1},\R_{2},\phi\right)$
is defined to be the subset of $\AR{K}{R_{1}}{R_{2}}{s}$, such that
for each pair $\left(\alpha^{\prime},\phi^{\prime}\right)\in\AR{K}{R_{1}}{R_{2}}{s}$,
the marked cells in row 1 and 2 of $\alpha^{\prime}$ are $\R_{1}$
and $\R_{2}$ respectively, $\alpha^{\prime}$ contains $w_{i,j}$
vertices in cell $\left(i,j\right)$, and $\phi^{\prime}=\phi$.
\end{defn}
Note that knowing $\mathbf{w}$, $\R_{1}$, $\R_{2}$ and $\phi$
is enough to determine whether an arrowed array satisfies the balance,
non-empty, or full conditions. It is also sufficient to determine
whether a vertex is critical, regardless of the actual pairing of
the vertices. Therefore, we can use these terms, and terms such as
arrow-head, arrow-tail, and points to with respect to $\Gamma$.

Next, we will lay the groundwork for the enumeration of arrowed arrays
satisfying a given substructure $\Gamma$. This involves introducing
several lemmas that limit the number of possibilities we have to consider,
as well as lemmas that allow us to remove pairings from arrowed arrays.
This allows us to categorize $\Gamma$ based on a number of parameters
that serve as invariants for the number of arrowed arrays that satisfy
$\Gamma$.
\begin{lem}
\label{lemma:Arrow Simplification Gamma-1}Let $\Gamma=\left(\mathbf{w},\R_{1},\R_{2},\phi\right)$
be a substructure of $\AR{K}{R_{1}}{R_{2}}{s}$, and suppose that
$\phi$ contains a column $\X$ that points to a column $\Y$, with
cell $\left(1,\Y\right)$ marked. Let $\Gamma^{\prime}=\left(\mathbf{w},\R_{1}\cup\left\{ \X\right\} ,\R_{2},\phi^{\prime}\right)$
be a substructure of $\AR{K}{R_{1}+1}{R_{2}}{s}$, such that
\begin{eqnarray*}
\phi^{\prime}\left(j\right) & = & \begin{cases}
\mbox{undefined} & j=\X\\
\phi\left(j\right) & j\in\H\backslash\X\mbox{ },
\end{cases}
\end{eqnarray*}
that is, instead of pointing to $\Y$, we mark cell $\left(1,\X\right)$
of $\Gamma^{\prime}$. Then, the number of arrowed arrays satisfying
$\Gamma$ and the number of arrowed arrays satisfying $\Gamma^{\prime}$
are equal. Furthermore, $\Gamma$ satisfies the balance, non-empty,
and full conditions if and only if $\Gamma^{\prime}$ satisfies them,
respectively.
\end{lem}
\begin{proof}
Let $\alpha\in\VA{K}{R_{1},R_{2}}{s}$ be a two-row vertical array,
and $\alpha^{\prime}$ be a vertical array otherwise identical to
$\alpha$, but with cell $\left(1,\X\right)$ marked. As we have not
changed the vertex pairings, $\psi_{2}$ remains unchanged between
$\left(\alpha,\phi\right)$ and $\left(\alpha^{\prime},\phi^{\prime}\right)$.
The only change to the functional digraph of $\psi_{1}$ is that $\X$
is also a root vertex, instead of simply pointing to one. Therefore,
$\left(\alpha,\phi\right)$ satisfies the forest condition if and
only if $\left(\alpha^{\prime},\phi^{\prime}\right)$ does, so the
number of arrowed arrays satisfying $\Gamma$ and $\Gamma^{\prime}$
are equal. As we have not changed the number of objects in each cell,
we see that $\Gamma$ satisfies the balance, non-empty, and full conditions
if and only if $\Gamma^{\prime}$ satisfies them, respectively.
\end{proof}
\begin{figure}
\begin{centering}
\includegraphics{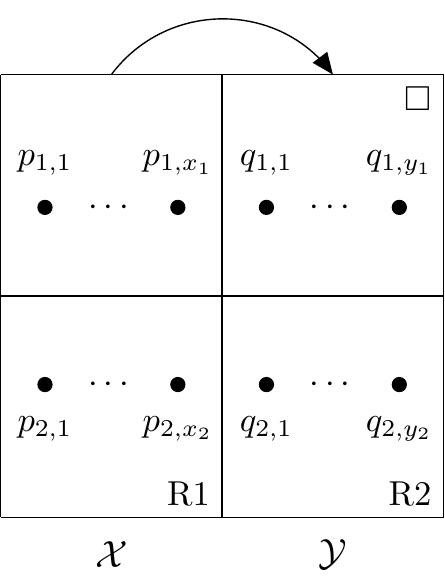}$\qquad\qquad\qquad$\includegraphics{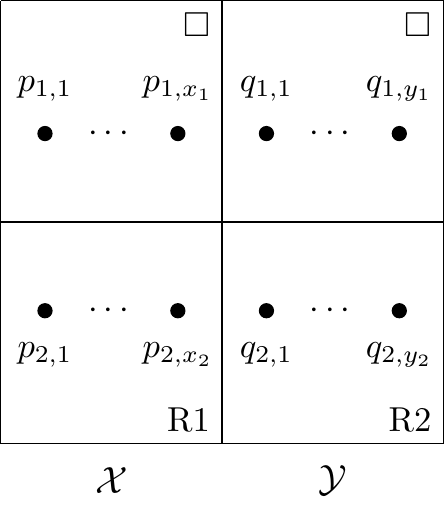}
\par\end{centering}
By applying the arrow simplification procedure to the left figure,
we arrive at the right figure. R1 and R2 can be arbitrary in whether
they are marked, but they must be the same between the two figures.

\caption{\label{fig:Arrow Simplification-1}Arrow Simplification 1}
\end{figure}

\begin{lem}
\label{lemma:Arrow Simplification Gamma-2}Let $\Gamma=\left(\mathbf{w},\R_{1},\R_{2},\phi\right)$
be a substructure of $\AR{K}{R_{1}}{R_{2}}{s}$, and suppose that
$\phi$ contains a column $\X$ that points to a column $\Y$, and
the column $\Y$ points to another column $\Z$. Let $\Gamma^{\prime}=\left(\mathbf{w},\R_{1},\R_{2},\phi^{\prime}\right)$
be a substructure of $\AR{K}{R_{1}}{R_{2}}{s}$ such that
\begin{eqnarray*}
\phi^{\prime}\left(j\right) & = & \begin{cases}
\Z & j=\X\\
\phi\left(j\right) & j\in\H\backslash\X\mbox{ },
\end{cases}
\end{eqnarray*}
that is, instead of pointing to $\Y$, $\X$ now points to $\Z$ in
$\phi^{\prime}$. Then, the number of arrowed arrays satisfying $\Gamma$
and the number of arrowed arrays satisfying $\Gamma^{\prime}$ are
equal. Furthermore, $\Gamma$ satisfies the balance, non-empty, and
full conditions if and only if $\Gamma^{\prime}$ satisfies them,
respectively.
\end{lem}
\begin{proof}
Let $\alpha\in\VA{K}{R_{1},R_{2}}{s}$ be a two-row vertical array.
Again, as we have not changed the vertex pairings, $\psi_{2}$ remains
unchanged between $\left(\alpha,\phi\right)$ and $\left(\alpha^{\prime},\phi^{\prime}\right)$.
The only change to the functional digraph of $\psi_{1}$ is that $\X$
now points to $\Z$, instead of pointing to $\Y$. This is the same
as detaching the subtree rooted at $\X$ from $\Y$, and attaching
it elsewhere on the same tree. Therefore, $\left(\alpha,\phi\right)$
satisfies the forest condition if and only if $\left(\alpha^{\prime},\phi^{\prime}\right)$
does, so the number of arrowed arrays satisfying $\Gamma$ and $\Gamma^{\prime}$
are equal. Again, as we have not changed the number of objects in
each cell, we see that $\Gamma$ satisfies the balance, non-empty,
and full conditions if and only if $\Gamma^{\prime}$ satisfies them,
respectively.
\end{proof}
\begin{figure}
\begin{centering}
\includegraphics{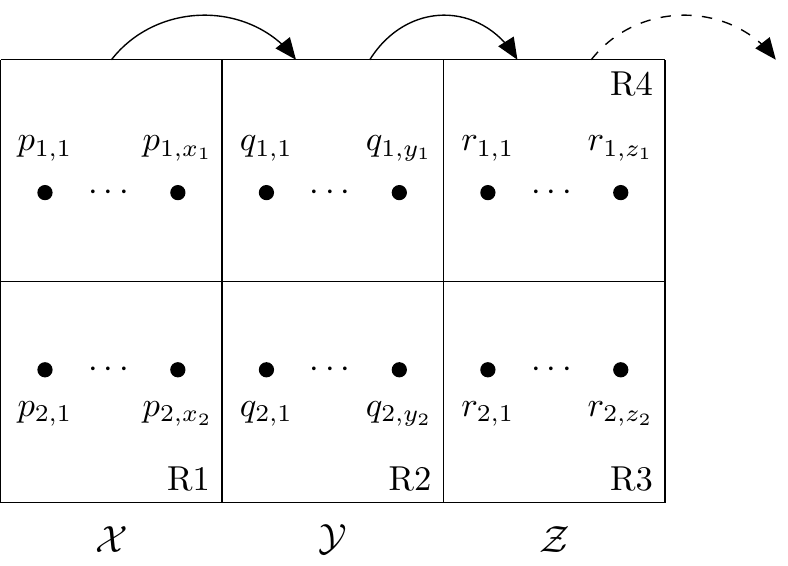}
\par\end{centering}
\begin{centering}
\includegraphics{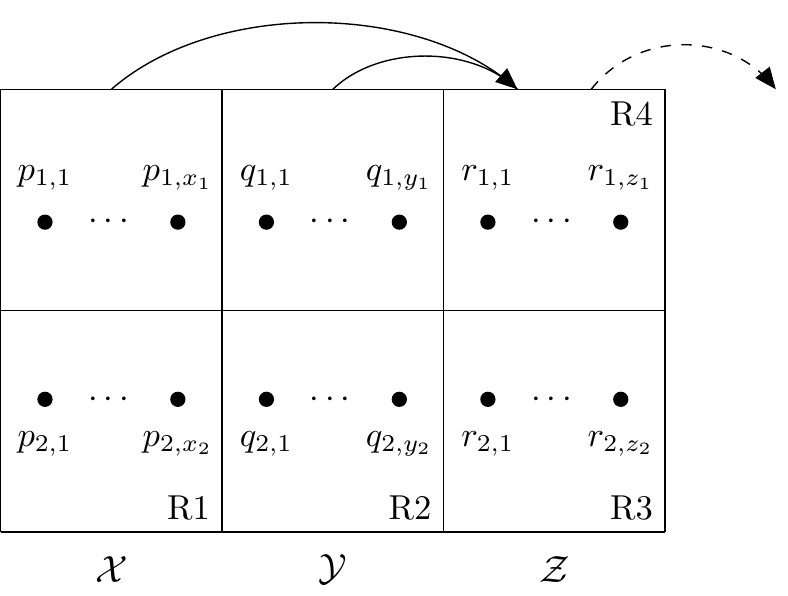}
\par\end{centering}
By applying the arrow simplification procedure to the top figure,
we arrive at the bottom figure. R1, R2, R3, and R4 can be arbitrary
in whether they are marked, but they must be the same between the
two figures. The same holds for the optional arrow with $\Z$ as its
tail.

\caption{\label{fig:Arrow Simplification-2}Arrow Simplification 2}
\end{figure}

Collectively, \ref{lemma:Arrow Simplification Gamma-1} and \ref{lemma:Arrow Simplification Gamma-2}
are the \emph{arrow simplification lemmas}, and pictures describing
the applications of these lemmas can be found in \ref{fig:Arrow Simplification-1}
and \ref{fig:Arrow Simplification-2}. Furthermore, applying these
lemmas to the array in \ref{fig:Arrowed Array} gives us \ref{fig:Irreducible Arrowed Array}.
Note that these lemmas can be applied repeatedly to simplify a substructure,
until either all arrow-heads are in cells that are unmarked and have
no arrow-tails, or an arrow-head is in the same cell as its own arrow-tail.
We are only interested in the former, as the latter implies that there
is a cycle in the functional digraph of $\phi$, which violates the
forest condition. This gives rise to the following definition.

\begin{figure}
\begin{centering}
\includegraphics{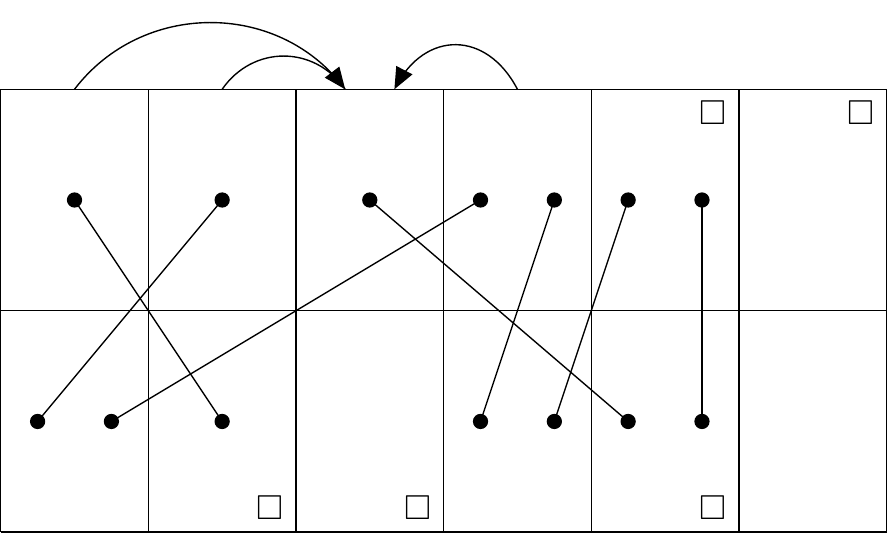}
\par\end{centering}
\caption{\label{fig:Irreducible Arrowed Array}Simplification of the arrowed
array in \ref{fig:Arrowed Array} into an irreducible array}
\end{figure}

\begin{defn}
\label{def:Irreducible Substructure Gamma}A substructure $\Gamma=\left(\mathbf{w},\R_{1},\R_{2},\phi\right)$
is \emph{irreducible} if the functional digraph of $\phi$ is acyclic,
and $\Gamma$ cannot be further simplified with the application of
the arrow simplification lemmas. Any cell of an irreducible substructure
containing an arrow-head must be unmarked in row 1, and cannot contain
an arrow-tail. Furthermore, it follows from definition that if an
irreducible substructure satisfies the full condition, then any cell
containing an arrow-head must also contain a critical vertex in row
1.

\label{def:Column type definition}If $\Gamma=\left(\mathbf{w},\R_{1},\R_{2},\phi\right)$
is an irreducible substructure, then we can categorize the columns
of $\Gamma$ as follows: Let $\A,\B,\C,\D$ be a partition of the
columns of $\K\backslash\H$, where

\begin{itemize}
\item Columns in $\A$ have both row 1 and row 2 unmarked
\item Columns in $\B$ have row 1 marked and row 2 unmarked
\item Columns in $\C$ have row 1 unmarked and row 2 marked
\item Columns in $\D$ have both row 1 and row 2 marked
\end{itemize}
Furthermore, if $\X$ is a column or a set of columns, let $\XB$
and $\XP$ be the sets of columns that have arrows pointing to $\X$,
and that have row 2 unmarked and marked, respectively. In particular,
$\AB$ and $\AP$ denotes the sets of columns pointing to $\A$, and
$\CB$ and $\CP$ denotes the sets of columns pointing to $\C$, with
row 2 unmarked and marked, respectively. These sets of columns implicitly
defined by $\Gamma$ are referred to as \emph{column types}, and a
diagram with all the column types can be found in \ref{fig:Column Types}.
\end{defn}
\begin{figure}
\begin{centering}
\includegraphics{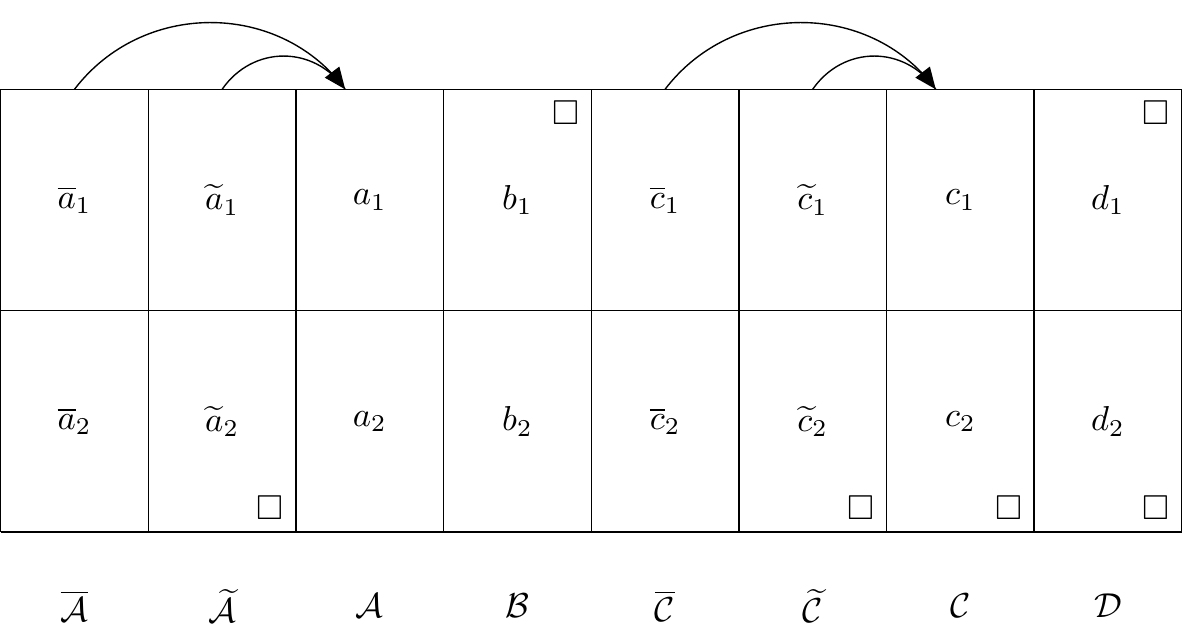}
\par\end{centering}
\caption{\label{fig:Column Types}Column types and variables for the number
of vertices}
\end{figure}

These eight column types form a partition of $\K$ on irreducible
substructures, and knowing the number of columns and the number of
vertices for each column type of $\Gamma$ is sufficient to count
the number of arrowed arrays satisfying it. However, before proving
the theorem for the number of arrowed arrays satisfying $\Gamma$,
we will need another two lemmas for simplifying arrowed arrays that
contain a fixed pair of vertices.
\begin{lem}
\label{lemma:Column Pointing Simplification}(column pointing) Let
$\Gamma=\left(\mathbf{w},\R_{1},\R_{2},\phi\right)$ be a substructure
of $\AR{K}{R_{1}}{R_{2}}{s}$, $v$ be a critical vertex in cell $\left(1,\X\right)$,
$u$ be a non-critical vertex in cell $\left(2,\Y\right)$, and $\X\neq\Y$.
Let the substructure $\Gamma_{vu}$ be the set of arrowed arrays that
satisfies $\Gamma$ and contains the pair $\left\{ v,u\right\} $,
and $\Gamma^{\prime}=\left(\mathbf{w}^{\prime},\R_{1},\R_{2},\phi^{\prime}\right)$
be a substructure of $\AR{K}{R_{1}}{R_{2}}{s-1}$ such that 
\begin{eqnarray*}
w_{i,j}^{\prime} & = & \begin{cases}
w_{i,j}-1 & \mbox{cell \ensuremath{\left(i,j\right)}contains \ensuremath{u} or \ensuremath{v}}\\
w_{i,j} & \mbox{otherwise}
\end{cases}\\
\phi^{\prime}\left(j\right) & = & \begin{cases}
\phi\left(j\right) & j\in\H\\
\Y & j=\X
\end{cases}
\end{eqnarray*}
Note that $\phi^{\prime}$ contains one more element in its domain
than $\phi$. Then, the number of arrowed arrays satisfying $\Gamma_{vu}$
and the number of arrowed arrays satisfying $\Gamma^{\prime}$ are
equal. Furthermore, $\Gamma_{vu}$ satisfies the non-empty and full
conditions if and only if $\Gamma^{\prime}$ satisfies them.
\end{lem}
\begin{proof}
To prove that the number of arrowed arrays are equal, we provide a
bijection between arrowed arrays satisfying $\Gamma_{vu}$ and arrowed
arrays satisfying $\Gamma^{\prime}$. Let $\left(\alpha,\phi\right)$
be an arrowed array that satisfies $\Gamma$ and contains the pair
$\left\{ v,u\right\} $. As $u$ is not critical, removing the pair
$\left\{ v,u\right\} $ does not affect $\psi_{2}$. Therefore, we
can obtain an arrowed array $\left(\alpha^{\prime},\phi^{\prime}\right)$
by removing $\left\{ v,u\right\} $ and replacing it by an arrow pointing
from $\X$ to $\Y$, while keeping all the other pairs intact. This
reduces the number of vertices in $\left(1,\X\right)$ and $\left(2,\Y\right)$
by 1, and leaves $\psi_{1}$ unchanged. Hence, the forest condition
is preserved, and $\left(\alpha^{\prime},\phi^{\prime}\right)$ satisfies
$\Gamma^{\prime}$.

Conversely, given an arrowed array $\left(\alpha^{\prime},\phi^{\prime}\right)$
that satisfies $\Gamma^{\prime}$, we can remove the arrow pointing
from $\X$ to $\Y$ and replace it by the pair $\left\{ v,u\right\} $
given by $\Gamma_{vu}$. Since the positions of $v$ and $u$ are
fixed in $\Gamma_{vu}$, there is no ambiguity as to where to add
them. Again, the forest condition is preserved as $\psi_{1}$ and
$\psi_{2}$ are unchanged by this substitution. Finally, both cells
$\left(1,\X\right)$ and $\left(2,\Y\right)$ contain at least one
object in both $\Gamma_{vu}$ and $\Gamma^{\prime}$. Cell $\left(1,\X\right)$
contains either a critical vertex or an arrow-tail, and cell $\left(2,\Y\right)$
contains at least one other object as $u$ is not critical. Since
all other cells remain unchanged, $\Gamma_{vu}$ satisfies the non-empty
and full conditions if and only if $\Gamma^{\prime}$ satisfies them.
\end{proof}
\begin{lem}
\label{lemma:Column Merging Simplification}(column merging) Let $\Gamma=\left(\mathbf{w},\R_{1},\R_{2},\phi\right)$
be a substructure of $\AR{K}{R_{1}}{R_{2}}{s}$, $v$ be a critical
vertex in cell $\left(1,\X\right)$, $u$ be a critical vertex in
cell $\left(2,\Y\right)$, and $\X\neq\Y$. Suppose that $\Gamma$
satisfies the full condition, and without loss of generality, assume
that $\Y$ is the last column of $\Gamma$ for purposes of column
indexing. Let the substructure $\Gamma_{vu}$ be the set of arrowed
arrays that satisfies $\Gamma$ and contains the pair $\left\{ v,u\right\} $,
and $\Gamma^{\prime}=\left(\mathbf{w}^{\prime},\R_{1}^{\prime},\R_{2}^{\prime},\phi^{\prime}\right)$
be a substructure of $\AR{K-1}{R_{1}}{R_{2}}{s-1}$ such that 
\begin{eqnarray*}
\R_{i}^{\prime} & = & \begin{cases}
\R_{i}\cup\X\backslash\Y & \Y\in\R_{i}\\
\R_{i} & \mbox{otherwise}
\end{cases}\\
w_{i,j}^{\prime} & = & \begin{cases}
w_{i,j}+w_{i,\Y}-1 & j=\X\\
w_{i,j} & \mbox{otherwise}
\end{cases}\\
\phi^{\prime}\left(j\right) & = & \begin{cases}
\phi\left(\Y\right) & j=\X,\phi\left(\Y\right)\mbox{ is defined}\\
\X & j\in\H,\phi\left(j\right)=\Y\\
\phi\left(j\right) & j\in\H,\phi\left(j\right)\ne\Y
\end{cases}
\end{eqnarray*}
Then, the number of arrowed arrays satisfying $\Gamma_{vu}$ and the
number of arrowed arrays satisfying $\Gamma^{\prime}$ are equal.
Furthermore, $\Gamma^{\prime}$ also satisfies the full condition.
\end{lem}
\begin{proof}
To prove that the number of arrowed arrays are equal, we provide a
bijection between arrowed arrays satisfying $\Gamma_{vu}$ and arrowed
arrays satisfying $\Gamma^{\prime}$. The idea behind this bijection
is to merge the columns $\X$ and $\Y$ in such a way that keeps the
rightmost objects of cell $\left(2,\X\right)$ and $\left(1,\Y\right)$
intact. As all other cells remain unchanged, $\Gamma^{\prime}$ satisfies
the full condition.

Let $\left(\alpha,\phi\right)$ be an arrowed array that satisfies
$\Gamma$ and contains the pair $\left\{ v,u\right\} $. To obtain
$\alpha^{\prime}$, we take the vertices of cell $\left(2,\Y\right)$
and place them in cell $\left(2,\X\right)$ in order, before the vertices
originally in $\left(2,\X\right)$. Then, for any column $j$ that
points to $\Y$, we change them to point to $\X$ instead. Similarly,
we take the vertices of cell $\left(1,\Y\right)$ and place them in
cell $\left(1,\X\right)$, but after the vertices originally in $\left(1,\X\right)$.
Furthermore, we mark cell $\left(1,\X\right)$ if cell $\left(1,\Y\right)$
is marked, and make $\X$ point to a column $\Z$ if column $\Y$
points to $\Z$ originally. Finally, we remove the pair $\left\{ v,u\right\} $
and the column $\Y$. Conversely, given an arrowed array $\left(\alpha^{\prime},\phi^{\prime}\right)$
that satisfies $\Gamma^{\prime}$, we can recover $\left(\alpha,\phi\right)$
by simply reversing the steps. As the arrows in row 1 and the number
of vertices in each cell is given by $\Gamma$, the reverse is unambiguous.

By construction, $\left(\alpha,\phi\right)$ satisfies $\Gamma_{vu}$
if and only if $\left(\alpha^{\prime},\phi^{\prime}\right)$ satisfies
$\Gamma^{\prime}$, with the possible exception of the forest condition.
Now, the critical pair $\left\{ u,v\right\} $ gives the edge $\left(\X,\Y\right)$
in the functional digraph of $\psi_{1}$, and the edge $\left(\Y,\X\right)$
in the functional digraph of $\psi_{2}$. By merging these two columns,
we are contracting these two edge in their respective functional digraph.
Therefore, $\psi_{i}$ satisfies the forest condition if and only
if $\psi_{i}^{\prime}$ satisfies it, for $i=1,2$. This shows that
the numbers of arrowed arrays satisfying $\Gamma_{vu}$ and $\Gamma^{\prime}$
are equal.
\end{proof}
The application of \ref{lemma:Column Pointing Simplification} to
replace $\Gamma_{vu}$ with $\Gamma^{\prime}$ is called the \emph{column
pointing procedure}, and a diagram of this procedure can be found
in \ref{fig:Column Pointing}. Similarly, the application of \ref{lemma:Column Merging Simplification}
to replace $\Gamma_{vu}$ with $\Gamma^{\prime}$ is called the \emph{column
merging procedure}, and a diagram of this procedure can be found in
\ref{fig:Column Merging}. After applying either procedure, we can
apply the arrow simplification lemmas to $\Gamma^{\prime}$ to further
simplify the substructure.

Note that unlike the other simplification lemmas, column merging requires
the substructure to satisfy the full condition. In particular, it
requires each cell of the columns being merged to be non-empty. Otherwise,
the resulting column will completely drop out of the forest condition,
which can break the bijection.

\begin{figure}
\begin{centering}
\includegraphics{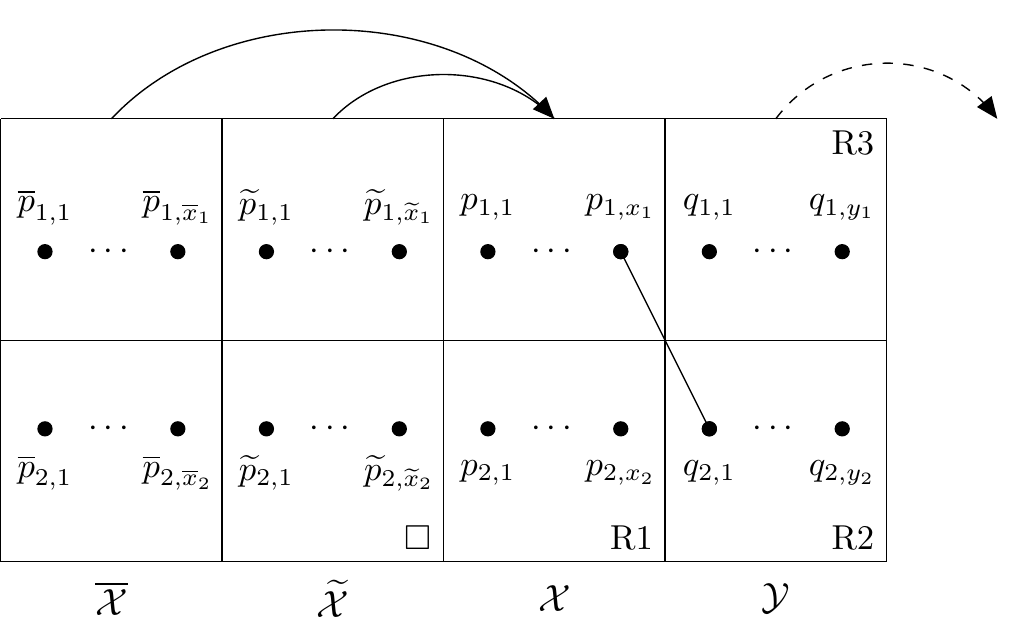}
\par\end{centering}
\begin{centering}
\includegraphics{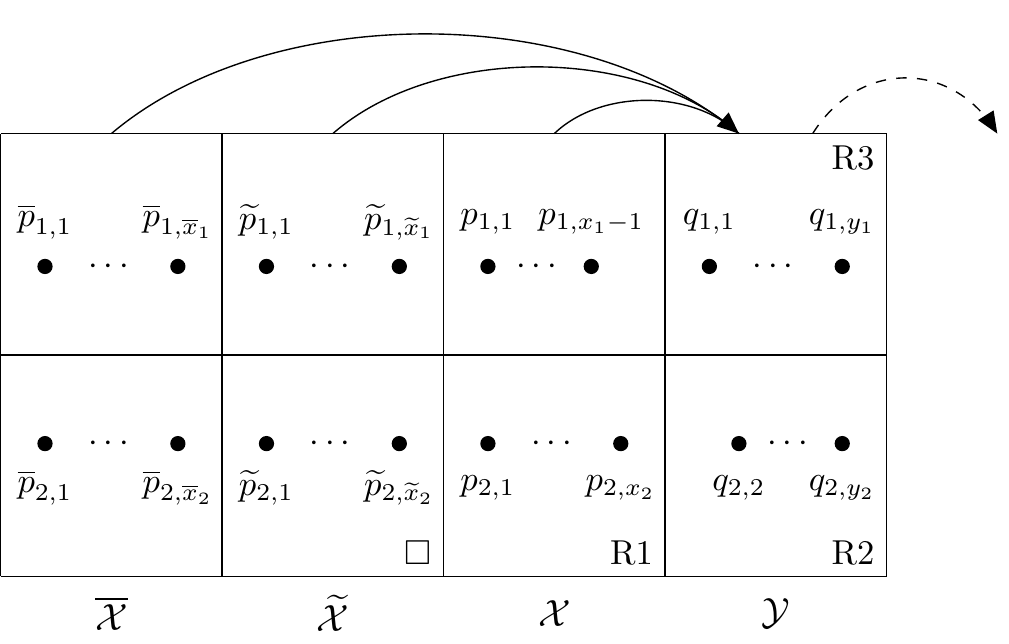}
\par\end{centering}
By applying the column pointing procedure to the top figure, we arrive
at the bottom figure. Here, $u=p_{1,x_{1}}$ and $v=q_{2,1}$. R1,
R2, and R3 can be arbitrary in whether they are marked, but they must
be the same between the two figures. The same holds for the optional
arrow with $\Y$ as its tail.

\caption{\label{fig:Column Pointing}Column pointing}
\end{figure}

\begin{figure}
\begin{centering}
\includegraphics{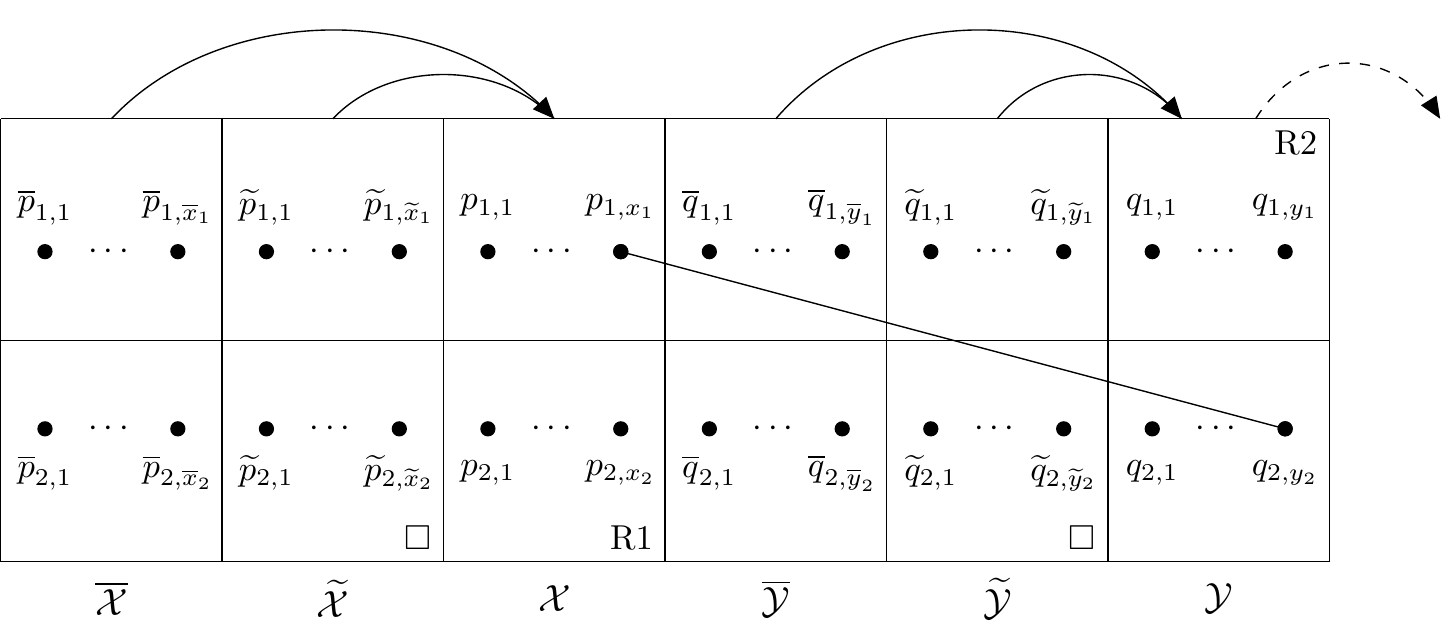}
\par\end{centering}
\begin{centering}
\includegraphics{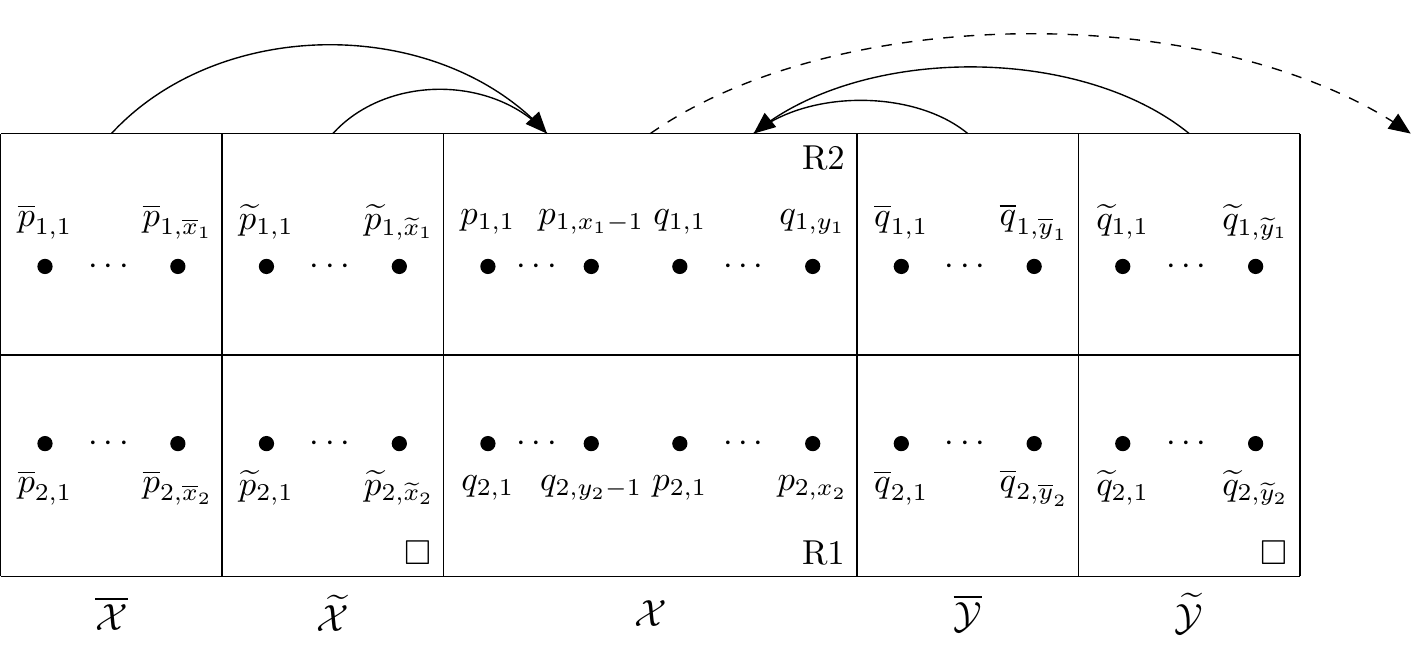}
\par\end{centering}
By applying the column merging procedure to the top figure, we arrive
at the bottom figure. Here, $u=p_{1,x_{1}}$ and $v=q_{2,y_{2}}$.
R1 and R2 can be arbitrary in whether they are marked, but they must
be the same between the two figures. The same holds for the optional
arrow with $\Y$ as its tail.

\caption{\label{fig:Column Merging}Column merging}
\end{figure}

\section{\label{sec:Substructure Gamma Formula}Enumeration of Substructure
\texorpdfstring{$\Gamma=\left(\mathbf{w},\R_{1},\R_{2},\phi\right)$}{Gamma}}

Now, we have everything we need to provide a formula for the number
of arrowed arrays satisfying the substructure $\Gamma=\left(\mathbf{w},\R_{1},\R_{2},\phi\right)$,
where $\Gamma$ is an irreducible substructure satisfying the full
condition. The formula will be given by the number of vertices in
each column type, as well as the number of columns of type $\A$.
Let $T\left(\Gamma\right)$ be the number of arrowed arrays that satisfy
the substructure $\Gamma$, then the following theorem gives the formula
for $T\left(\Gamma\right)$.
\begin{thm}
\label{thm:Substructure general formula}Given an irreducible substructure
$\Gamma=\left(\mathbf{w},\R_{1},\R_{2},\phi\right)$ that satisfies
the full condition with $s\ge A+2$, the number of arrowed arrays
$\left(\alpha,\phi\right)\in\AR{K}{R_{1}}{R_{2}}{s}$ that satisfy
$\Gamma$ is given by the formula
\[
T\left(\Gamma\right)=\left(s-1\right)!\left[\frac{\left(b_{2}+d_{2}\right)\left(\ap_{1}+c_{1}+\cp_{1}+d_{1}\right)}{s-A}+\frac{b_{1}\left(c_{2}+\cb_{2}+\cp_{2}\right)-\cb_{1}\left(b_{2}+d_{2}\right)}{\left(s-A\right)\left(s-A-1\right)}\right]
\]
In the case where $s=A+1$, the formula reduces to 
\[
T\left(\Gamma\right)=\left(s-1\right)!\left(b_{2}+d_{2}\right)\left(\ap_{1}+c_{1}+\cp_{1}+d_{1}\right)
\]
\end{thm}
By the convention set out in \ref{conv:Array Convention}, we let
a lower case variable $x_{i}$ represent the total number of points
in row $i$ of the columns of type $\X$, and $A$ represent the number
of columns of type $\A$.
\begin{proof}
We prove this via induction on the total number of vertices, and tiebreak
by the number of critical vertices in row 2. There are two base cases
and three inductive cases to consider, depending on whether $\Gamma$
contains a column of type $\A$, a column of type $\C$ and no columns
of type $\A$, or no columns of type $\A$ or $\C$. Also, we will
only do the proof for $s\ge A+2$. In the case where $s=A+1$, the
proof is the same, but we have to use the second formula to avoid
division by zero.

\textbf{Base case 1:}

Suppose $\Gamma$ has no critical vertex. As $\Gamma$ is irreducible,
each cell must either be marked or have an arrow-tail. However, the
latter cannot happen as an arrow-head of an irreducible substructure
must be in an unmarked cell. Hence, every cell of $\Gamma$ must be
marked, so the forest condition is trivially satisfied. Therefore,
there are $s!$ ways to pair the vertices of the array. By substituting
$d_{1}=d_{2}=s$ into $T\left(\Gamma\right)$, and setting all other
variables to 0, we see that $T\left(\Gamma\right)=s!$ as desired.

\textbf{Base case 2:}

If $s=2$, $\A=\emptyset$, and $\C\neq\emptyset$, then 
\begin{eqnarray*}
T\left(\Gamma\right) & = & \left[\frac{\left(b_{2}+d_{2}\right)\left(c_{1}+\cp_{1}+d_{1}\right)}{2}+\frac{b_{1}\left(c_{2}+\cb_{2}+\cp_{2}\right)-\cb_{1}\left(b_{2}+d_{2}\right)}{2}\right]\\
 & = & b_{1}+\left(1-b_{1}-\cb_{1}\right)\left(b_{2}+d_{2}\right)
\end{eqnarray*}
by substituting in $2=b_{i}+\cb_{i}+\cp_{i}+c_{i}+d_{i}$. This case
is needed as the inductive step for $\Gamma$ containing no columns
of type $\A$ but at least one column of type $\C$ requires that
$T\left(\Gamma\right)$ be true for $s-1$. However, if $s=1$, then
$s<A+2$, and this creates a zero in the denominator of our formula.
The formula can be proved by checking all possible positions of the
vertices in row 1. The details are omitted as it is tedious and not
enlightening.

\textbf{Case 1:}

Suppose $\Gamma$ contains at least one column of type $\A$, and
$\X$ is one such column. Let $\XB$ and $\XP$ be columns pointing
to $\X$ as defined in \ref{def:Column type definition}, and note
that they are columns of type $\AB$ and $\AP$, respectively. Then,
the critical vertex $v$ of cell $\left(1,\X\right)$ must be paired
with some vertex $u$ in a cell $\left(2,\Y\right)$. To satisfy the
forest condition for row 1, $\Y$ cannot be a column of $\X$, $\XB$,
or $\XP$. By fixing $u$, we can pair vertices $u$ and $v$ to obtain
the substructure $\Gamma_{uv}$. Then, we simplify $\Gamma_{uv}$
using the column pointing and column merging procedures described
in \ref{lemma:Column Pointing Simplification} and \ref{lemma:Column Merging Simplification},
which makes the columns of $\X$, $\XB$, and $\XP$ point to $\Y$.
Now, $\Y$ cannot point to $\X$, $\XB$, or $\XP$, as that would
either imply that $\Y\in\XB\cup\XP$, or that $\Gamma$ is not irreducible.
Therefore, $\Y$ must either not contain an arrow-tail, or be pointing
to some other column $\Z$ that has a critical vertex in row 1. Therefore,
the functional digraph of $\phi$ is acyclic, and by using the arrow
simplification procedures described in \ref{lemma:Arrow Simplification Gamma-1}
and \ref{lemma:Arrow Simplification Gamma-2}, we obtain an irreducible
substructure $\Gamma^{\prime}$ that has one less vertex per row than
$\Gamma$. Furthermore, both $s$ and $A$ decrease by 1, so the inequality
$s\ge A+2$ holds. Depending on the column type of $\Y$ and whether
$u$ is critical, we can use the inductive hypothesis to determine
$T\left(\Gamma^{\prime}\right)$ in terms of existing parameters given
by the column types of $\Gamma$.

For example, let $\Y$ be a column of type $\D$. Then, after applying
the column pointing procedure, $\X$ becomes a column of type $\B$,
the columns of $\XB$ become columns of type $\B$, and the columns
of type $\XP$ become columns of type $\D$. Hence, in the resulting
substructure $\Gamma^{\prime}=\Gamma_{\A\D}$ after simplification,
we have

\begin{itemize}
\item $a_{i}^{\prime}=a_{i}-x_{i}$
\item $\ab_{i}^{\prime}=\ab_{i}-\xb_{i}$
\item $\ap_{i}^{\prime}=\ap_{i}-\xp_{i}$
\item $b_{i}^{\prime}=b_{i}+x_{i}+\xb_{i}-\delta_{1,i}$
\item $d_{i}^{\prime}=d_{i}+\xp_{i}-\delta_{2,i}$
\end{itemize}
where $\delta_{i,j}=1$ if $i=j$, and 0 otherwise. Substituting this
into the inductive hypothesis, we have 
\begin{eqnarray*}
T\left(\Gamma_{\A\D}\right) & = & \left(s-2\right)!\left[\frac{\left(b_{2}+x_{2}+\xb_{2}+d_{2}+\xp_{2}-1\right)\left(\ap_{1}+c_{1}+\cp_{1}+d_{1}\right)}{s-A}+\right.\\
 &  & \left.\frac{\left(b_{1}+x_{1}+\xb_{1}-1\right)\left(c_{2}+\cb_{2}+\cp_{2}\right)-\cb_{1}\left(b_{2}+x_{2}+\xb_{2}+d_{2}+\xp_{2}-1\right)}{\left(s-A\right)\left(s-A-1\right)}\right]
\end{eqnarray*}

Similarly, we define $T\left(\Gamma_{\A\A}\right)$ and $T\left(\Gamma_{\A\C}\right)$
to be the number of arrowed arrays satisfying substructure $\Gamma^{\prime}$
if $v$ is in a column of type $\A$ and $\C$, respectively. Then,
we repeat this computation for the remaining possible column types
of $\Y$, and whether $u$ is critical. These are given by the column
types $\A$, $\AB$, $\AP$, $\B$, $\C$, $\CB$, and $\CP$. In
the cases of $\A$, $\AB$, $\B$, and $\CB$, the particular substitutions
are dependent on whether $v$ is also critical, even though the formulas
for $T\left(\Gamma^{\prime}\right)$ are the same. Furthermore, these
can all be expressed in terms of $T\left(\Gamma_{\A\A}\right)$, $T\left(\Gamma_{\A\C}\right)$,
and $T\left(\Gamma_{\A\D}\right)$. By letting $u$ range across all
vertices of row 2, we obtain all possible pairings of the critical
vertex $v$ in column $\X$. Therefore, by counting the number of
vertices of each column type, we obtain the number of occurrences
of each $\Gamma^{\prime}$. Adding everything together, we have 
\begin{eqnarray*}
T\left(\Gamma\right) & = & \left(a_{2}-x_{2}+\ab_{2}-\xb_{2}+\ap_{2}-\xp_{2}\right)T\left(\Gamma_{\A\A}\right)+\\
 &  & \left(c_{2}+\cb_{2}+\cp_{2}\right)T\left(\Gamma_{\A\C}\right)+\left(b_{2}+d_{2}\right)T\left(\Gamma_{\A\D}\right)
\end{eqnarray*}
By substituting in $s=a_{i}+\ab_{i}+\ap_{i}+c_{i}+\cb_{i}+\cp_{i}+b_{i}+d_{i}$
and simplifying, we can show that $T\left(\Gamma\right)$ satisfies
the inductive hypothesis. This proves the case where $\Gamma$ contains
a column of type $\A$.

\textbf{Case 2:}

Suppose that $\Gamma$ does not contain columns of type $\A$, but
contains at least one column of type $\C$. The formula simplifies
to 
\[
T\left(\Gamma\right)=\left(s-1\right)!\left[\frac{\left(b_{2}+d_{2}\right)\left(c_{1}+\cp_{1}+d_{1}\right)}{s}+\frac{b_{1}\left(c_{2}+\cb_{2}+\cp_{2}\right)-\cb_{1}\left(b_{2}+d_{2}\right)}{s\left(s-1\right)}\right]
\]

While the formula is simpler in this case, the proof is slightly more
involved. Let $\X$ be a fixed column of type $\C$, and let $\XB$
and $\XP$ be columns pointing to $\X$ as defined in \ref{def:Column type definition}.
Note that they are columns of type $\CB$ and $\CP$, respectively.
As in Case 1, the critical vertex $v$ of cell $\left(1,\X\right)$
must be paired with some vertex $u$ in a cell. Again, to satisfy
the forest condition for row 1, $\Y$ cannot be a column of $\X$,
$\XB$ or, $\XP$. Therefore, we pair $u$ and $v$ to obtain the
substructure $\Gamma_{uv}$, which we simplify using the same lemmas
used in Case 1 to obtain an irreducible substructure $\Gamma^{\prime}$.
As the case $s=2$ is already handled, we can assume $s\ge3$, so
$s\ge A+2$ still holds. Depending on the column type of $\Y$ and
whether $u$ is critical, we can use the inductive hypothesis to determine
$T\left(\Gamma^{\prime}\right)$ in terms of existing parameters given
by column types of $\Gamma$. The major difference in this case is
that if $u$ is a critical vertex, then both $\X$ and $\Y$ become
columns of a different type, so we must introduce the parameters $y_{i}$
for the number of vertices in column $i$ of $\Y$.

As in Case 1, we define $T\left(\Gamma_{\C\B}\right)$ and $T\left(\Gamma_{\C\C}\right)$
to be the number of arrowed arrays satisfying substructure $\Gamma^{\prime}$
if $v$ is in a column of type $\B$ and $\C$, respectively. However,
we also need the correction terms $T_{\C\B c}$ and $T_{\C\CB c}$
for the cases of $\B$, and $\CB$, depending on whether the vertex
$v$ is critical. Then, we can compute $T\left(\Gamma^{\prime}\right)$
for all possible column types of $\Y$, and whether $u$ is critical.
These are given by the column types $\B$, $\C$, $\CB$, $\CP$,
and $\D$, and can all be expressed in terms of $T\left(\Gamma_{\A\B}\right)$,
$T\left(\Gamma_{\A\C}\right)$, $T_{\C\B c}$, and $T_{\C\CB c}$.

By letting $u$ range across all vertices of row 2, we obtain all
possible pairings of the critical vertex $v$ in column $\X$. Notice
that as we pair $v$ each vertex of $\B$, we add $y_{1}T_{\C\B c}$
if and only if $u$ is the rightmost vertex of $\Y$. Since each column
of $\B$ has exactly one rightmost vertex, $\sum_{\Y\in\B}y_{1}=b_{1}$.
Similarly, $\sum_{\Y\in\CB}y_{1}=\cb_{1}-\xb_{1}$. Therefore, by
counting the number of vertices of each column type, we obtain the
number of occurrences of each $\Gamma^{\prime}$. Adding everything
together, we have 
\begin{eqnarray*}
T\left(\Gamma\right) & = & \left(c_{2}-x_{2}+\cb_{2}-\xb_{2}+\cp_{2}-\xp_{2}\right)T\left(\Gamma_{\C\C}\right)+\\
 &  & \left(\cb_{1}-\xb_{1}\right)T_{\C\CB c}+\left(b_{2}+d_{2}\right)T\left(\Gamma_{\C\B}\right)+b_{1}T_{\C\B c}
\end{eqnarray*}
By substituting in $s=c_{2}+\cb_{2}+\cp_{2}+b_{2}+d_{2}$ and simplifying,
we can show that $T\left(\Gamma\right)$ satisfies the inductive hypothesis.
This proves the case where $\Gamma$ contains a column of type $\C$,
but no columns of type $\A$.

\textbf{Case 3:}

If $\Gamma$ does not contain any column of type $\A$ or $\C$, then
every cell in row 1 is marked, leaving us only with columns of type
$\B$ and $\D$. In this case, the formula simplifies to 
\[
T\left(\Gamma\right)=d_{1}\left(s-1\right)!
\]
as $s=b_{2}+d_{2}$. Since $\Gamma$ does not contain any arrows,
we can switch the two rows and invert the roles of $\B$ and $\C$
to obtain $\Gamma^{\prime}$. Furthermore, at least one cell in row
2 is unmarked, as otherwise we would have the base case. Therefore,
the number of critical vertices in row 2 decreases in $\Gamma^{\prime}$,
and we can continue the induction using Case 2. Furthermore, neither
$s$ nor $A$ changed, so $s\ge A+2$ still holds. Now, $\Gamma^{\prime}$
only have columns of type $\C$ and $\D$, so by the inductive hypothesis,
\[
T\left(\Gamma^{\prime}\right)=d_{2}\left(s-1\right)!
\]
as $s=c_{1}+d_{1}$ in $\Gamma^{\prime}$. This completes the induction
and proves our formula for $T\left(\Gamma\right)$.
\end{proof}
Note that if $\Gamma$ satisfies the full condition and $s\le A$,
then $T\left(\Gamma\right)=0$, as each column of type $\A$ requires
one critical vertex for each row. Furthermore, as those vertices can
only be paired with each other, $\psi_{i}\left(\X\right)\in\A$ for
all $\X\in\A$. This violates the forest condition for row $i$.
\begin{cor}
\label{cor:Substructure general formula}Given a substructure $\Gamma=\left(\mathbf{w},\R_{1},\R_{2},\phi\right)$
where $\phi$ is empty and $s\ge A+2$, the number of arrowed arrays
$\left(\alpha,\phi\right)\in\AR{K}{R_{1}}{R_{2}}{s}$ that satisfy
$\Gamma$ is given by the formula
\[
T\left(\Gamma\right)=\left(s-1\right)!\left[\frac{\left(b_{2}+d_{2}\right)\left(c_{1}+d_{1}\right)}{s-A}+\frac{b_{1}c_{2}}{\left(s-A\right)\left(s-A-1\right)}\right]
\]
where $A$ is the number of columns that contains no marked cells
and at least one vertex in each row. In the case where $s=A+1$, the
formula simplifies to 
\[
T\left(\Gamma\right)=\left(s-1\right)!\left[\frac{\left(b_{2}+d_{2}\right)\left(c_{1}+d_{1}\right)}{s-A}\right]
\]
\end{cor}
Note that \ref{cor:Substructure general formula} holds even if the
full condition is not satisfied, and the definition of $A$ has been
adjusted to match this. This stems from the fact that we can remove
columns with no arrows or vertices without impacting the forest condition. 

\section{\label{sec:Two Row Vertical}Enumerating Proper Vertical Arrays}

Finally, we are ready to compute the formula for $v_{K;R_{1},R_{2}}^{\left(s\right)}$
using arrowed arrays. As proper vertical arrays are arrowed arrays
that satisfies the non-empty, balance, and forest conditions that
contain no arrows, we can take $\phi$ to be empty and $\mathbf{w}$
to be a vector of size $K$. To enumerate proper vertical arrays,
we will define a coarser substructure, which we will compute the formula
for using our formula of $T\left(\Gamma\right)$.
\begin{defn}
\label{def:Substructure Omega}Let $\mathbf{w}$ be a non-negative
vector. The substructure $\Omega=\left(\mathbf{w}\right)$ is defined
to be the subset of $\PVA{K}{R_{1},R_{2}}{s}$ that satisfies the
non-empty and balance conditions, such that for each pair $\alpha\in\PVA{K}{R_{1},R_{2}}{s}$,
$\alpha$ contains $w_{j}$ vertices in both cells $\left(1,j\right)$
and $\left(2,j\right)$. For a given substructure $\Omega=\left(\mathbf{w}\right)$
and $A\ge0$, we define $\Omega_{A}$ to be the substructure that
describes the set of arrowed arrays that satisfies $\Omega$, and
have exactly $A$ (non-empty) columns of type $\A$. For convenience,
we say a substructure $\Gamma$ is a \emph{refinement} of another
substructure $\Omega$ if the set of arrowed arrays satisfying $\Gamma$
is a subset of the arrowed arrays satisfying $\Omega$. We denote
it as $\Gamma\hookrightarrow\Omega$. Furthermore, if $\Gamma_{1},\dots,\Gamma_{t}$
is a set of substructures that are refinements of a substructure $\Omega$,
we say that $\Gamma_{1},\dots,\Gamma_{t}$ \emph{partitions} $\Omega$
if the sets of arrowed arrays satisfying the $\Gamma_{i}$'s are mutually
disjoint, and their union is the set of arrowed arrays that satisfy
$\Omega$.
\end{defn}
By considering all possible $R_{1}$-subsets $\R_{1}$ and $R_{2}$-subsets
$\R_{2}$, we see that the set of substructures of the form $\Gamma=\left(\left[\mathbf{w},\mathbf{w}\right],\R_{1},\R_{2},\emptyset\right)$
partition the substructure $\Omega$. Furthermore, the subset of substructures
with exactly $A$ columns of type $\A$ partitions $\Omega_{A}$,
which in turn partitions $\Omega$ by taking $A$ from 0 to $s-1$.
With the substructure $\Omega=\left(\mathbf{w}\right)$ defined, we
will now provide a formula for it, which we will use to decompose
vertical arrays into arrowed arrays.
\begin{thm}
\label{thm:Substructure Omega General Formula}Let $R_{1},R_{2}\ge1$,
and let $\Omega=\left(\mathbf{w}\right)$ be a substructure with $F$
columns that contains vertices, denoted $\F$. Then, the number of
vertical arrays $\alpha\in\PVA{K}{R_{1},R_{2}}{s}$ satisfying the
substructure $\Omega$ is given by the formula
\[
T\left(\Omega\right)=s!\sum_{A=0}^{s-1}\frac{s}{s-A}\binom{F-1}{A}\binom{K-A-1}{K-A-R_{1},K-A-R_{2},R_{1}+R_{2}-K+A-1}
\]
where $\binom{a+b+c}{a,b,c}=\frac{\left(a+b+c\right)!}{a!b!c!}$ is
the multinomial coefficient.
\end{thm}
\begin{proof}
To prove this theorem, we sum $T\left(\Omega\right)$ over all substructures
$\Gamma=\left(\left[\mathbf{w},\mathbf{w}\right],\R_{1},\R_{2},\emptyset\right)$
that are refinements of $\Omega$. Note that $T\left(\Gamma\right)$
as given in \ref{cor:Substructure general formula} only depends on
the number non-empty of columns of type $\A$, even though it depends
on the number of vertices of other column types. Therefore, we first
sum over all $\Gamma$ with $A$ non-empty columns of type $\A$ to
obtain $T\left(\Omega_{A}\right)$, then we sum $A$ from 0 to $s-1$
to obtain $T\left(\Omega\right)$. As $\Omega$ satisfies the balance
condition, so must all $\Gamma$ that are refinements of $\Omega$.
This implies that we can drop the subscripts from $T\left(\Gamma\right)$.
For convenience, we will refer to the number of vertices of row 1
in a set of column $\X$ simply as the number of vertices in $\X$,
as that number is the same between row 1 or row 2.

Now, let $\Gamma$ be a refinement of $\Omega$, and suppose $\Gamma$
have $A$, $B$, $C$, and $D$ columns of type $\A$, $\B$, $\C$,
and $\D$, respectively. Then, as the columns marked in row 1 are
type $\B$ and $\D$, and the columns marked in row 2 are type $\C$
and $\D$, we have
\begin{eqnarray*}
B & = & K-A-R_{2}\\
C & = & K-A-R_{1}\\
D & = & R_{1}+R_{2}-K+A
\end{eqnarray*}
Therefore, there are 
\[
\binom{F}{A}\binom{K-A}{K-A-R_{2},K-A-R_{1},R_{1}+R_{2}-K+A}
\]
substructures $\Gamma$ that are refinements of $\Omega_{A}$. Note
that the columns of type $\A$ must be non-empty, as they must be
a subset of the columns of $\F$.

Now, we can rewrite $T\left(\Gamma\right)$ as
\begin{eqnarray*}
T\left(\Gamma\right) & = & \left(s-1\right)!\left[\frac{bd}{s-A}+\frac{cd}{s-A}+\frac{d^{2}}{s-A}+\frac{bc}{s-A-1}\right]\\
 & = & \left(s-1\right)!\left(T_{1}\left(\Gamma\right)+T_{2}\left(\Gamma\right)+T_{3}\left(\Gamma\right)+T_{4}\left(\Gamma\right)\right)
\end{eqnarray*}
for $0\le A\le s-2$. For $A=s-1$, we let $T_{4}\left(\Gamma\right)=bc=0$,
as $A=s-1$ means there are $s-1$ columns of type $\A$, which means
that the remaining non-empty column cannot be both type $\B$ and
type $\C$ at the same time. As the substructures $\Gamma$ with $A$
columns of type $\A$ partitions $\Omega_{A}$, we can let $T_{i}\left(\Omega_{A}\right)=\sum_{\Gamma\hookrightarrow\Omega_{A}}T_{i}\left(\Gamma\right)$
for $i=1,2,3,4$, which gives us 
\[
T\left(\Omega\right)=\left(s-1\right)!\left(\sum_{A=0}^{s-1}T_{1}\left(\Omega_{A}\right)+\sum_{A=0}^{s-1}T_{2}\left(\Omega_{A}\right)+\sum_{A=0}^{s-1}T_{3}\left(\Omega_{A}\right)+\sum_{A=0}^{s-2}T_{4}\left(\Omega_{A}\right)\right)
\]

Now, let $\left\{ v,u\right\} $ be a pair of vertices such that $v$
and $u$ are in cell $\left(1,\X\right)$ and cell $\left(2,\Y\right)$,
respectively. If $\X$ and $\Y$ are distinct columns, then $\left(v,u\right)$
contributes to $T_{1}\left(\Gamma\right)$ if and only if $\X$ is
of type $\C$ and $\Y$ is of type $\B$. To have $A$ columns of
$\A$, exactly $A$ of the remaining $F-2$ columns of $\F$ must
be unmarked. Then, the remaining columns must be of $\B$, $\C$,
and $\D$, which can be arbitrarily chosen from the remaining $K-A-2$
columns. If we let $w=w_{1}^{2}+\cdots+w_{K}^{2}$ be the squares
of the number of vertices in each column, then we have 
\[
T_{1}\left(\Omega_{A}\right)=\frac{s^{2}-w}{s-A}\binom{F-2}{A}\binom{K-A-2}{K-A-R_{1},K-A-R_{2}-1,R_{1}+R_{2}-K+A-1}
\]
Similar calculations give us 
\begin{eqnarray*}
T_{2}\left(\Omega_{A}\right) & = & \frac{s^{2}-w}{s-A}\binom{F-2}{A}\binom{K-A-2}{K-A-R_{1}-1,K-A-R_{2},R_{1}+R_{2}-K+A-1}\\
T_{4}\left(\Omega_{A}\right) & = & \frac{s^{2}-w}{s-A-1}\binom{F-2}{A}\binom{K-A-2}{K-A-R_{1}-1,K-A-R_{2}-1,R_{1}+R_{2}-K+A}
\end{eqnarray*}
To obtain $T_{3}\left(\Omega_{A}\right)$, we break it up into 2 cases,
depending on whether $\X=\Y$. If $\X=\Y$, we have
\[
T_{3a}\left(\Omega_{A}\right)=\frac{w}{s-A}\binom{F-1}{A}\binom{K-A-2}{K-A-R_{1},K-A-R_{2},R_{1}+R_{2}-K+A-1}
\]
Otherwise, we get 
\[
T_{3b}\left(\Omega_{A}\right)=\frac{s^{2}-w}{s-A}\binom{F-2}{A}\binom{K-A-2}{K-A-R_{1},K-A-R_{2},R_{1}+R_{2}-K+A-2}
\]

To sum over $A$, we shift the index of $T_{4}\left(\Omega_{A}\right)$
by 1, and observe that for $1\le A\le s-1$, we have 
\begin{eqnarray*}
 &  & T_{1}\left(\Omega_{A}\right)+T_{2}\left(\Omega_{A}\right)+T_{3a}\left(\Omega_{A}\right)+T_{3b}\left(\Omega_{A}\right)+T_{4}\left(\Omega_{A-1}\right)\\
 & = & \frac{s^{2}}{s-A}\binom{F-1}{A}\binom{K-A-1}{K-A-R_{1},K-A-R_{2},R_{1}+R_{2}-K+A-1}
\end{eqnarray*}
which is independent of $w$. Furthermore, for $A=0$, we have
\begin{eqnarray*}
 &  & T_{1}\left(\Omega_{0}\right)+T_{2}\left(\Omega_{0}\right)+T_{3}\left(\Omega_{0}\right)\\
 & = & s\binom{K-1}{K-R_{1},K-R_{2},R_{1}+R_{2}-K-1}
\end{eqnarray*}
which is in agreement with the previous sum. Therefore, we have the
formula for $T\left(\Omega\right)$ as 
\[
T\left(\Omega\right)=\sum_{A=0}^{s-1}\frac{s\cdot s!}{s-A}\binom{F-1}{A}\binom{K-A-1}{K-A-R_{1},K-A-R_{2},R_{1}+R_{2}-K+A-1}
\]
\end{proof}
To obtain the formula for the number of vertical arrays in Goulden
and Slofstra, we need to sum over all possible ways of placing $s$
points into $K$ columns. Doing so gives us the following theorem.
\begin{thm}
\label{thm:Vertical Array Formula}Let $s,K,R_{1},R_{2}\ge1$. Then,
\begin{eqnarray*}
v_{K;R_{1},R_{2}}^{\left(s\right)} & = & \frac{\left(s+R_{1}-1\right)!\left(s+R_{2}-1\right)!}{\left(s+R_{1}+R_{2}-2\right)!}\cdot\binom{s+R_{1}+R_{2}-2}{K-1}\times\\
 &  & \left[\binom{K-1}{R_{1}-1}\binom{K-1}{R_{2}-1}-\binom{K-1}{s+R_{1}-1}\binom{K-1}{s+R_{2}-1}\right]
\end{eqnarray*}
\end{thm}
\begin{proof}
For $F\ge0$, there are $\binom{K}{F}$ to choose $\F$ columns so
that each of them contains at least one vertex, and there are $\binom{s-1}{F-1}$
ways to distribute $s$ vertices into those columns. Hence, the number
of proper vertical arrays satisfying the non-empty condition is
\begin{eqnarray*}
v_{K;R_{1},R_{2}}^{\left(s\right)} & = & \sum_{F\ge0}\binom{K}{F}\binom{s-1}{F-1}T\left(\Omega\right)\\
 & = & s!\sum_{A=0}^{s-1}\sum_{F\ge0}\binom{K}{F}\binom{s}{A}\binom{s-A-1}{s-F}\binom{K-A-1}{K-A-R_{1},K-A-R_{2},R_{1}+R_{2}-K+A-1}\\
 & = & \sum_{A=0}^{s-1}\binom{s}{A}\frac{\left(s+K-A-1\right)!}{\left(K-A-R_{1}\right)!\left(K-A-R_{2}\right)!\left(R_{1}+R_{2}-K+A-1\right)!}
\end{eqnarray*}
by the Chu-Vandermonde identity (pg. 67 of \cite{Andrews-Askey-Roy:1999}).
As the binomial coefficient $\binom{s}{A}$ implies that the natural
upper bound of the sum is $A$, we can rewrite this sum as 
\begin{eqnarray*}
v_{K;R_{1},R_{2}}^{\left(s\right)} & = & \sum_{A\ge0}\binom{s}{A}\frac{\left(s+K-A-1\right)!}{\left(K-A-R_{1}\right)!\left(K-A-R_{2}\right)!\left(R_{1}+R_{2}-K+A-1\right)!}\\
 &  & -\frac{\left(K-1\right)!}{\left(K-s-R_{1}\right)!\left(K-s-R_{2}\right)!\left(R_{1}+R_{2}-K+s-1\right)!}
\end{eqnarray*}
Now, by the Pfaff-Saalschütz identity (pg. 69 of \cite{Andrews-Askey-Roy:1999}),
we can rewrite the first part as 
\begin{eqnarray*}
 &  & \sum_{A\ge0}\binom{s}{A}\frac{s!\left(s+K-A-1\right)!}{\left(K-A-R_{1}\right)!\left(K-A-R_{2}\right)!\left(R_{1}+R_{2}-K+A-1\right)!}\\
 & = & _{3}F_{2}\left({-s,-K+R_{1},-K+R_{2}\atop R_{1}+R_{2}-K,-s-K+1};1\right)\frac{s!\left(s+K-1\right)!}{\left(K-R_{1}\right)!\left(K-R_{2}\right)!\left(R_{1}+R_{2}-K-1\right)!}\\
 & = & \frac{s!\left(R_{1}+s-1\right)!\left(R_{2}+s-1\right)!\left(K-1\right)!}{\left(R_{1}-1\right)!\left(R_{2}-1\right)!\left(R_{1}+R_{2}-K+s-1\right)!\left(K-R_{1}\right)!\left(K-R_{2}\right)!}
\end{eqnarray*}
Combined together, this gives the formula for $v_{K;R_{1},R_{2}}^{\left(s\right)}$
that appears in the Goulden and Slofstra paper.
\end{proof}

\section{\label{sec:Goulden Slofstra Reduction}Further Reduction to the Goulden-Slofstra
Formula}

In this section, we will show a method of reducing the number of sums
in the formula of Goulden and Slofstra using Pfaff's identity. We
start by rewriting \ref{thm:Goulden-Slofstra} as $A_{2}^{\left(q_{1},q_{2};s\right)}\left(x\right)=g_{1}-g_{2}$
using our notation, where
\begin{eqnarray*}
g_{1} & = & \sum_{k=1}^{d+1}\sum_{t_{1}\ge0}\sum_{t_{2}\ge0}\frac{\left(2q_{1}+s\right)!\left(2q_{2}+s\right)!\left(k-1\right)!}{2^{t_{1}+t_{2}}t_{1}!t_{2}!\left(d-t_{1}-t_{2}-k+1\right)!}\cdot\binom{x}{k}\times\\
 &  & \frac{1}{\left(q_{1}-t_{1}\right)!\left(k-q_{1}+t_{1}-1\right)!\left(q_{2}-t_{2}\right)!\left(k-q_{2}+t_{2}-1\right)!}\\
g_{2} & = & \sum_{k=1}^{d+1}\sum_{t_{1}\ge0}\sum_{t_{2}\ge0}\frac{\left(2q_{1}+s\right)!\left(2q_{2}+s\right)!\left(k-1\right)!}{2^{t_{1}+t_{2}}t_{1}!t_{2}!\left(d-t_{1}-t_{2}-k+1\right)!}\cdot\binom{x}{k}\times\\
 &  & \frac{1}{\left(s+q_{1}-t_{1}\right)!\left(k-s-q_{1}+t_{1}-1\right)!\left(s+q_{2}-t_{2}\right)!\left(k-s-q_{2}+t_{2}-1\right)!}
\end{eqnarray*}
with $d=q_{1}+q_{2}+s$ as in the original theorem. Note that we have
removed the upper bounds for $t_{1}$ and $t_{2}$, as the summation
terms can only be non-zero if both $t_{1}\le q_{1}$ and $t_{2}\le q_{2}$
hold. To reduce the number of sums in $g_{1}$ and $g_{2}$, we manipulate
them separately with the same transforms. We first use Pfaff's identity
to transform the sum involving $t_{1}$, then use the Chu-Vandermonde
identity to eliminate $t_{2}$. Afterwards, we make the summation
variables symmetric by making a substitution for $k$, before combining
the results together. For reference, the identities used for this
procedure can be found in pg. 67 and pg. 69 of \cite{Andrews-Askey-Roy:1999}.

By rewriting the $t_{1}$ sum of $g_{1}$ using the standard notation
for hypergeometric series and using Pfaff's identity, we have

\begin{eqnarray*}
g_{1} & = & \sum_{k=1}^{d+1}\sum_{t_{2}\ge0}\frac{1}{2^{t_{2}}t_{2}!}\cdot\binom{x}{k}{}_{2}F_{1}\left({-d+t_{2}+k-1,-q_{1}\atop k-q_{1}};\frac{1}{2}\right)\times\\
 &  & \frac{\left(2q_{1}+s\right)!\left(2q_{2}+s\right)!\left(k-1\right)!}{\left(d-t_{2}-k+1\right)!q_{1}!\left(k-q_{1}-1\right)!\left(q_{2}-t_{2}\right)!\left(k-q_{2}+t_{2}-1\right)!}\\
 & = & \sum_{k=1}^{d+1}\sum_{t_{2}\ge0}\frac{1}{2^{t_{2}}t_{2}!}\cdot\binom{x}{k}\left(1-\frac{1}{2}\right)^{d-t_{2}-k+1}{}_{2}F_{1}\left({-d+t_{2}+k-1,k\atop k-q_{1}};-1\right)\times\\
 &  & \frac{\left(2q_{1}+s\right)!\left(2q_{2}+s\right)!\left(k-1\right)!}{\left(d-t_{2}-k+1\right)!q_{1}!\left(k-q_{1}-1\right)!\left(q_{2}-t_{2}\right)!\left(k-q_{2}+t_{2}-1\right)!}\\
 & = & \sum_{k=1}^{d+1}\sum_{t_{2}\ge0}\sum_{t_{1}\ge0}\frac{1}{2^{d-k+1}t_{1}!t_{2}!\left(d-t_{1}-t_{2}-k+1\right)!}\cdot\binom{x}{k}\times\\
 &  & \frac{\left(2q_{1}+s\right)!\left(2q_{2}+s\right)!\left(k+t_{1}-1\right)!}{q_{1}!\left(k-q_{1}+t_{1}-1\right)!\left(q_{2}-t_{2}\right)!\left(k-q_{2}+t_{2}-1\right)!}
\end{eqnarray*}
While there is no upper bound for $t_{1}$, the term $\left(d-t_{1}-t_{2}-k+1\right)!$
in the denominator causes the sum to terminate. Furthermore, for the
summation term to be non-zero, we must have $d-t_{1}-t_{2}-k+1\ge0$
and $k-q_{2}+t_{2}-1\ge0$ at the same time. Combining these inequalities
together gives us $t_{1}\le q_{1}+s$, which can be used as an upper
bound for $t_{1}$. Next, we rewrite the $t_{2}$ sum as a hypergeometric
series, and note that it satisfies the Chu-Vandermonde identity. This
yields,
\begin{eqnarray*}
g_{1} & = & \sum_{k=1}^{d+1}\sum_{t_{1}=0}^{q_{1}+s}\frac{1}{2^{d-k+1}t_{1}!}\cdot\binom{x}{k}{}_{2}F_{1}\left({-q_{2},-d+t_{1}+k-1\atop k-q_{2}};1\right)\times\\
 &  & \frac{\left(2q_{1}+s\right)!\left(2q_{2}+s\right)!\left(k+t_{1}-1\right)!}{\left(d-t_{1}-k+1\right)!q_{1}!\left(k-q_{1}+t_{1}-1\right)!q_{2}!\left(k-q_{2}-1\right)!}\\
 & = & \sum_{k=1}^{d+1}\sum_{t_{1}=0}^{q_{1}+s}\frac{\left(d-t_{1}\right)!}{2^{d-k+1}t_{1}!\left(d-t_{1}-k+1\right)!}\cdot\binom{x}{k}\times\\
 &  & \frac{\left(2q_{1}+s\right)!\left(2q_{2}+s\right)!\left(k+t_{1}-1\right)!}{q_{1}!\left(s+q_{1}-t_{1}\right)!\left(k-q_{1}+t_{1}-1\right)!q_{2}!\left(k-1\right)!}
\end{eqnarray*}
Note that the term $\left(d-t_{1}-k+1\right)!$ in the denominator
means that for $k>d-t_{1}+1$, the summation term is zero. Therefore,
we can switch the two sums and lower the upper bound of $k$ to $d-t_{1}+1$.
Next, the terms $\left(k-q_{1}+t_{1}-1\right)!$ and $\left(k-1\right)!$
in the denominator means that for the summand to be non-zero, we have
$k\ge\max\left\{ q_{1}-t_{1}+1,1\right\} $. Hence, we can change
the lower bound of $k$ to $q_{1}-t_{1}+1$. As $k+t_{1}-1\ge q_{1}\ge0$
with this new lower bound, the factorial term in the numerator remains
non-negative. After changing the bounds, we can reverse the sum with
the substitution $k=d-t_{1}-t_{2}+1$. This gives us the formula
\begin{eqnarray}
g_{1} & = & \sum_{t_{1}=0}^{q_{1}+s}\sum_{t_{2}=0}^{q_{2}+s}\frac{\left(d-t_{1}\right)!\left(d-t_{2}\right)!\left(2q_{1}+s\right)!\left(2q_{2}+s\right)!}{2^{t_{1}+t_{2}}t_{1}!t_{2}!\left(d-t_{1}-t_{2}\right)!}\cdot\binom{x}{d-t_{1}-t_{2}+1}\times\label{eq:g1 formula}\\
 &  & \frac{1}{q_{1}!q_{2}!\left(s+q_{1}-t_{1}\right)!\left(s+q_{2}-t_{2}\right)!}\nonumber 
\end{eqnarray}
which is symmetric between $t_{1}$ and $t_{2}$.

We now apply the same transformations to $g_{2}$. However, instead
of changing the upper bound to $t_{1}\le q_{1}+s$, we have $t_{1}\le q_{1}$.
Then, after applying the Chu-Vandermonde identity, we can tighten
the bounds of $k$ to $q_{1}+s-t_{1}+1\le k\le d-t_{1}+1$. Finally,
we can reverse the sum with the substitution $k=d-t_{1}-t_{2}+1$.
This gives us the formula
\begin{eqnarray}
g_{2} & = & \sum_{t_{1}=0}^{q_{1}}\sum_{t_{2}=0}^{q_{2}}\frac{\left(d-t_{1}\right)!\left(d-t_{2}\right)!\left(2q_{1}+s\right)!\left(2q_{2}+s\right)!}{2^{t_{1}+t_{2}}t_{1}!t_{2}!\left(d-t_{1}-t_{2}\right)!}\cdot\binom{x}{d-t_{1}-t_{2}+1}\times\label{eq:g2 formula}\\
 &  & \frac{1}{\left(q_{1}+s\right)!\left(q_{2}+s\right)!\left(q_{1}-t_{1}\right)!\left(q_{2}-t_{2}\right)!}\nonumber 
\end{eqnarray}
which is again symmetric in $t_{1}$ and $t_{2}$.

As we have $\left(q_{1}-t_{1}\right)!$ and $\left(q_{2}-t_{2}\right)!$
in the denominator of $g_{2}$, we can actually increase the bounds
of $t_{1}$ and $t_{2}$ to $q_{1}+s$ and $q_{2}+s$ without changing
the sum, matching the bounds of $g_{1}$. Finally, we can put \ref{eq:g1 formula}
and \ref{eq:g2 formula} together and obtain 
\begin{eqnarray*}
A_{2}^{\left(q_{1},q_{2};s\right)}\left(x\right) & = & g_{1}-g_{2}\\
 & = & \sum_{t_{1}=0}^{q_{1}+s}\sum_{t_{2}=0}^{q_{2}+s}\frac{\left(d-t_{1}\right)!\left(d-t_{2}\right)!\left(2q_{1}+s\right)!\left(2q_{2}+s\right)!}{2^{t_{1}+t_{2}}t_{1}!t_{2}!\left(d-t_{1}-t_{2}\right)!}\cdot\binom{x}{d-t_{1}-t_{2}+1}\times\\
 &  & \left[\frac{1}{q_{1}!q_{2}!\left(s+q_{1}-t_{1}\right)!\left(s+q_{2}-t_{2}\right)!}-\frac{1}{\left(q_{1}+s\right)!\left(q_{2}+s\right)!\left(q_{1}-t_{1}\right)!\left(q_{2}-t_{2}\right)!}\right]
\end{eqnarray*}
where $d=q_{1}+q_{2}+s$.

\section{Acknowledgements}

Many thanks for the help of I.P. Goulden for supporting me in my doctoral
studies, during which this research is conducted, as well as the editing
and verifying of the results in this paper.

\bibliographystyle{plain}
\bibliography{uw-ethesis}

\begin{thebibliography}{10}

\bibitem{Andrews-Askey-Roy:1999}
G.E. Andrews, R.~Askey, and R.~Roy.
\newblock {\em Special Functions}.
\newblock Cambridge University Press, 1999.

\bibitem{ChanThesis:2016}
A.C.S. Chan.
\newblock {\em Combinatorial Methods for Enumerating Maps in Surfaces of
  Arbitrary Genus}.
\newblock PhD thesis, University of Waterloo, 2016.

\bibitem{Chan:2017-2}
A.C.S. Chan.
\newblock Enumeration of tree-like maps with arbitrary number of vertices.
\newblock forthcoming, 2017.

\bibitem{Goulden-Nica:2005}
I.P. Goulden and A.~Nica.
\newblock A direct bijection for the {H}arer-{Z}agier formula.
\newblock {\em Journal of Combinatorial Theory, Series A}, 111(2):224--238,
  August 2005.

\bibitem{Goulden-Slofstra:2010}
I.P. Goulden and W.~Slofstra.
\newblock Annular embeddings of permutations for arbitrary genus.
\newblock {\em Journal of Combinatorial Theory, Series A}, 117(3):272--288,
  April 2010.

\bibitem{Harer-Zagier:1986}
J.~Harer and D.~Zagier.
\newblock The {E}uler characteristic of the moduli space of curves.
\newblock {\em Inventiones Mathematicae}, 85:457--486, 1986.

\bibitem{Itzykson-Zuber:1990}
C.~Itzykson and J.-B. Zuber.
\newblock Matrix integration and combinatorics of modular groups.
\newblock {\em Communications in Mathematical Physics}, 134(3):197--207, 1990.

\bibitem{Jackson:1994}
D.M. Jackson.
\newblock On an integral representation for the genus series for 2-cell
  embeddings.
\newblock {\em Transactions of the American Mathematical Society},
  344(2):755--772, August 1994.

\bibitem{Kerov:1999}
S.~Kerov.
\newblock Rook placements on ferrer boards and matrix integrals.
\newblock {\em Journal of Mathematical Sciences}, 96(5):3531--3536, October
  1999.

\bibitem{Kontsevich:1992}
M.~Kontsevich.
\newblock Intersection theory on the moduli space of curves and matrix airy
  functions.
\newblock {\em Communications in Mathematical Physics}, 147:1--23, 1992.

\bibitem{Lando-Zvonkin:2004}
S.K. Lando and A.K. Zvonkin.
\newblock {\em Graphs on Surfaces and Their Applications}, volume 141 of {\em
  Encyclopaedia of Mathematical Sciences}.
\newblock Springer, 2004.

\bibitem{Lass:2001}
B.~Lass.
\newblock Démonstration combinatoire de la formule de {H}arer-{Z}agier.
\newblock {\em Comptes Rendus de l'Académie des Sciences, Series I},
  333:155--160, 2001.

\bibitem{Penner:1988}
R.C. Penner.
\newblock Perturbative series and the moduli space of {R}iemann surfaces.
\newblock {\em Journal of Differential Geometry}, 27:35--53, 1988.

\bibitem{Zagier:1995}
D.~Zagier.
\newblock On the distribution of the number of cycles of elements in symmetric
  groups.
\newblock {\em Nieuw archief voor wiskunde}, 13:489--495, 1995.

\end{thebibliography}

\end{document}